\documentclass[12pt]{article}
\usepackage{amsthm}
\usepackage{amsmath}
\usepackage{amsfonts}
\usepackage{amssymb}
\usepackage{amscd}
\usepackage{url}
\usepackage{epsfig}
\usepackage{lscape}
\usepackage{url}
\usepackage{color}
\usepackage{hyperref}
\date{}

\newtheorem{thm}{Theorem}[section]
\newtheorem{cor}[thm]{Corollary}
\newtheorem{lem}[thm]{Lemma}
\newtheorem{prop}[thm]{Proposition}

\newtheorem{ex}[thm]{Example}

\theoremstyle{definition}

\newtheorem{example}[thm]{Example}
\newtheorem{remark}[thm]{Remark}

\newcommand{\Z}{\mathbb{Z}}
\newcommand{\N}{\mathbb{N}}
\newcommand{\Q}{\mathbb{Q}}
\newcommand{\PP}{\mathbb{P}}
\newcommand{\TP}{\mathbb{T}{\mathbb{P}}}
\newcommand{\HF}{HF}
\newcommand{\Fl}{\operatorname{Fl}}
\newcommand{\Osh}{\mathcal O}

\begin{document}

\bibliographystyle{plain}
\title{\bf The Hilbert scheme of the diagonal \\
in a product of projective spaces}

\author{Dustin Cartwright  and  Bernd Sturmfels }

\maketitle

\begin{abstract}
\noindent The diagonal in a product of projective spaces
is cut out by the ideal of $2 {\times} 2$-minors of
a matrix of unknowns. The multigraded Hilbert scheme which classifies its
degenerations 
has a unique Borel-fixed ideal.
This Hilbert scheme is generally reducible, and its
main component is a compactification of
${\rm PGL}(d)^n/{\rm PGL}(d)$.
For $n=2$ we recover
the manifold of complete collineations. For
projective lines we obtain a novel space of trees
that is irreducible but singular.
All ideals in our Hilbert scheme are radical.
We also explore connections to affine buildings and Deligne schemes.
\end{abstract}

\section{Introduction}

Multigraded Hilbert schemes parametrize
families of ideals in a polynomial ring that share
the same Hilbert function with respect to some
grading by an abelian group~\cite{HS}. We are interested in the following
particular case.
Let $X = (x_{ij})$ be a $d {\times} n$-matrix of unknowns.
We fix the polynomial ring $K[X]$ over a field $K$ with the
 $\Z^n$-grading by column degree, i.e.\ ${\rm deg}(x_{ij}) = e_j$.
In this grading, the Hilbert function of the
polynomial ring $K[X]$ equals
\begin{equation*}
\N^n  \rightarrow \N \,,\,\,
(u_1,\ldots,u_n)\, \mapsto\, \prod_{i=1}^n \binom{u_i + d-1}{d-1} .
\end{equation*}
We study the multigraded Hilbert scheme   $H_{d,n}$,
 which parametrizes $\Z^n$-homogeneous ideals $I$ in $K[X]$
such that $K[X]/I$ has the Hilbert function
\begin{equation}
\label{eqn:ourHF}
\N^n  \rightarrow \N \,,\,\,
(u_1,\ldots,u_n)\, \mapsto\, \binom{u_1 {+} u_2 + \cdots + u_n+d-1}{d-1}. 
\end{equation}
The key example is the ideal $I_2(X)$
that is generated by the  $2 {\times} 2$-minors of $X$,
and whose quotient is indeed $\Z^n$-graded
with Hilbert function~(\ref{eqn:ourHF}).
The Hilbert scheme $H_{d,n}$  has the following geometric interpretation.
Each $\Z^n$-homogeneous ideal in $K[X]$ specifies a subscheme
of the product of projective spaces  $(\PP^{d-1})^n = 
\PP^{d-1} \! \times \cdots \times \PP^{d-1}$. The subscheme
specified by the ideal $I_2(X)$ is the diagonal embedding of 
$\PP^{d-1}$ in $(\PP^{d-1})^n$. Our Hilbert scheme 
$H_{d,n}$ is a natural parameter space for
degenerations of this diagonal in $(\PP^{d-1})^n$.

The results obtained in this paper are as follows.
In Section 2 we prove that all ideals $I$ in $H_{d,n}$
are radical and Cohen-Macaulay. This result is derived 
by identifying a distinguished Borel-fixed ideal $Z$ 
with these properties. It confirms a conjecture
on multilinear Gr\"obner bases made by Conca in \cite{Conca}.

In Section 3 we show that $I_2(X)$ and one of its initial monomial ideals
are smooth points on $H_{d,n}$. The 
irreducible component containing these points is an
equivariant compactification of the homogeneous space $G^n/G$ where
$G = {\rm PGL}(d)$, and $G \subset G^n$ is the diagonal embedding. For $n = 2$
we recover Thaddeus' construction in  \cite{Th} of the space of complete
collineations.
The relationship of our compactification of $G^n/G$
to those constructed by
Lafforgue in \cite{Laf} will be discussed in
Remark \ref{remark:Lafforgue} and Example \ref{example:twothree}.

Section 4 is concerned with the case $d=2$, and we regard its results to be the
main contribution of this paper.
We show that $H_{2,n}$ is irreducible, but singular
for $n \geq 4$, and we determine its combinatorial structure.
Each point in $H_{2,n}$ corresponds to a certain tree of projective lines.
Among these are precisely $2^n (n+1)^{n-2} $ monomial ideals, one for each
tree on $n+1$ unlabeled vertices and $n$ labeled directed edges,
and these form a graph.

In Section 5 we study the case $d=n=3$. These
are the smallest parameters 
for which the 
multigraded Hilbert scheme $H_{d,n}$ is reducible.
We show that $H_{3,3}$
is the reduced union of seven irreducible components,
with the main component of dimension $16$ 
parametrizing degenerations of the diagonal
in $\PP^2 {\times} \PP^2 {\times} \PP^2$.
We list all monomial ideals on
$H_{3,3}$ and their incidence relations.

Section 6 outlines a connection to convexity in
affine buildings and tropical geometry. Extending
previous results in \cite{BY, KT}, we show how
Gr\"obner degenerations on $H_{d,n}$ can be used
to compute special fibers of Deligne schemes.

\section{On a conjecture of Conca}

Our plan is to derive
Conca's Conjecture~4.2 in~\cite{Conca}
from the following result.

\begin{thm}
\label{thm:radical}
All ideals $I$ corresponding to points in $H_{d,n}$ are radical ideals.
\end{thm}

\begin{proof}
We may assume that $K$ is an infinite field.
Let  $G = {\rm PGL}(d,K)$, the group of invertible
$d \times d$-matrices modulo scaling, let
 $B$ be the Borel subgroup of images of upper
triangular matrices in $G$, and  let $T$ be the algebraic torus
of images of diagonal matrices in $G$. Then $T^n$ is 
a maximal torus in $G^n$, and $B^n $ is a Borel subgroup in $G^n$.
We consider the action of these groups
on the Hilbert scheme $H_{d,n}$. The $T^n$-fixed points
of $H_{d,n}$ are the  monomial ideals that have
the same  $\Z^n$-graded Hilbert function as $I_2(X)$.
It suffices to assume that $I$ is such a monomial ideal
because every other ideal $J \in H_{d,n}$ can be degenerated to a
monomial ideal $I = {\rm in}(J)$ via Gr\"obner bases, 
and if ${\rm in}(J)$ is radical then so is $J$.

We can further assume that $I$ is {\em Borel-fixed},
which means that $I$ is fixed under the action of $B^n$.
Indeed, if $A_1, A_2, \ldots , A_n$
are generic matrices of $G$ then we replace the
ideal $I$ first by its image $I' = (A_1,A_2, \ldots,A_n) \circ I$,
and then by the initial monomial ideal ${\rm in}(I')$.
The ideal ${\rm in}(I') = {\rm gin}(I)$ is 
the {\em multigraded generic initial ideal}.
The same approach as in \cite[\S 15.9.2]{Eis} shows that
 ${\rm gin}(I)$ is Borel-fixed.
Moreover, if ${\rm gin}(I)$ is radical then so is $I$. Hence,
it suffices to  show that every
Borel-fixed ideal $I$ in $H_{d,n} $ is a radical ideal.

Our result will be a direct consequence of the following two claims:

\noindent Claim 1:
{\em There is precisely one Borel-fixed ideal $Z$ in $H_{d,n}$.}

\noindent Claim 2: {\em The unique Borel-fixed ideal $Z$ is radical.}

\smallskip

We first describe the ideal $Z$ and then prove that it has these 
 properties.
Let $u$ be any vector in the set
\begin{equation}
\label{eqn:ineqal}
U = \left\{(u_1, \ldots, u_n) \in \Z^n : 
0 \leq u_i \leq d-1 \hbox{ and } 
\textstyle\sum_i u_i = (n-1)(d-1)\right\}.
\end{equation}
We write $Z_u$ for the ideal generated by all unknowns
 $x_{ij}$ with $i \leq u_j$ and $1 \leq j \leq n$.
This is a Borel-fixed prime monomial ideal. 
Consider the intersection of the prime ideals $Z_u$:
\begin{equation*}
Z \,\,:= \,\,\bigcap_{u \in U} Z_u.
\end{equation*}
The monomial ideal $Z$ is a radical and Borel-fixed.
Each of its ${\binom{d+n-2}{d-1}}$ associated
prime ideals $Z_u$ has the same codimension $(n-1)(d-1)$.

We now apply Conca's results in \cite[Section 5]{Conca}.
He showed that $Z$ has the same Hilbert function as $I_2(X)$.
Therefore, the ideal $Z$ is the promised
Borel-fixed ideal in $H_{d,n}$.
More precisely, \cite[Theorem 5.1]{Conca}
states that $Z$ is precisely the
generic initial ideal ${\rm gin}(I_2(X))$ of
the ideal of $2 \times 2$-minors.

We claim that $Z$ is the only
Borel-fixed monomial ideal in $H_{d,n}$.
To show this, we apply results about
the {\em multidegree} in \cite[\S 8.5]{MS}.
The multidegree of the prime ideal $Z_u$ is the monomial 
$\, {\bf t}^u = t_1^{u_1} t_2^{u_2} \cdots t_n^{u_n} $.
By \cite[Theorem 8.44]{MS},
$Z_u$ is the only unmixed Borel-fixed monomial ideal
having multidegree ${\bf t}^u$.
By \cite[Theorem 8.53]{MS}, the Borel-fixed 
ideal $Z = \cap_u Z_u$ has the multidegree
\begin{equation}
\label{eqn:multidegofZ}
 \mathcal{C}\bigl(k[X]/Z;{\bf t}\bigr) 
\quad = \quad
 \sum_{u \in U} {\bf t}^{u}
\quad = \quad
 \sum_{u \in U} t_1^{u_1} t_2^{u_2} \cdots t_n^{u_n} .
\end{equation}
Since the multidegree of a homogeneous ideal is determined by 
its Hilbert series \cite[Claim 8.54]{MS}, we conclude that
every ideal $I \in H_{d,n}$ has multidegree~(\ref{eqn:multidegofZ}).
Now, suppose that $I \in H_{d,n}$ is Borel-fixed.
Since $I$ is monomial, each minimal primary component contributes at most one
term ${\bf t}^{u}$ to the multidegree~(\ref{eqn:multidegofZ}). Thus,
by \cite[Theorem 8.53]{MS},
the minimal primes of $I$ are precisely the prime ideals
$Z_u$ where $u$ runs over the elements of $U$.
This implies $\,I \subseteq \sqrt{I} = Z $. 
However, since  $I$ and $Z$ have
the same Hilbert function in a positive grading, we conclude
that $I = Z$, as desired.
\end{proof}

\begin{remark}
Our proof of Theorem \ref{thm:radical} was based
on an idea that was suggested to us
by Michel Brion. In~\cite{BrionMult}, Brion proves that for any
multiplicity-free subvariety of a flag variety, such as the
diagonal in a product of projective spaces, there exists a flat
degeneration to
a reduced union of Schubert varieties, which is our~$Z$. Although  
 Theorem \ref{thm:radical}
only applies to the special case of the diagonal in a product of
projective spaces, it establishes reducedness not just for
{\em some}
degeneration but for 
{\em any} ideal with the same multigraded Hilbert function.
Our proof
combined the nice argument from~\cite{BrionMult} with the
explicit description of the Borel-fixed monomial ideal given by Conca in~\cite{Conca}. \qed
\end{remark} 

We now come to the question asked by Conca in 
\cite[Conjecture 1.1]{Conca}. Given any $d \times d$-matrices
 $A_1, A_2, \ldots , A_n$ with entries in $K$, we apply them
individually to the $n$ columns of the matrix
$X = (x_{ij})$, form the ideal of $2 \times 2$-minors
of the resulting  $d\times n$-matrix, 
and then take  the initial monomial ideal
\begin{equation}
\label{eqn:concaideal}
\,{\rm in}((A_1,  \ldots, A_n) \circ I_2(X))
\end{equation}
with respect to some term order. Conca conjectures that
(\ref{eqn:concaideal}) is always a squarefree monomial ideal.
He proves this for generic $A_i$ by showing that 
(\ref{eqn:concaideal}) equals the Borel-fixed ideal $Z$
constructed above.  Theorem \ref{thm:radical} implies
the same conclusion under the much weaker hypothesis that the
 $A_i$ are invertible.

\begin{cor} \label{cor:itssquarefree}
For any invertible $d \times d$-matrices $A_1,  \ldots, A_n$ 
and any term order on $K[X]$, the monomial ideal
$\,{\rm in}((A_1,  \ldots, A_n) \circ I_2(X))$ is squarefree.
\end{cor}

\begin{proof}
Applying invertible matrices $A_i$ to $I_2(X)$
corresponds to taking the orbit of $I_2(X)$
under the action of $G^n $ on $H_{d,n}$.
Therefore, (\ref{eqn:concaideal}) is a monomial ideal
that lies in $H_{d,n}$. By Theorem \ref{thm:radical},
it is radical and hence squarefree.
\end{proof}

Corollary~\ref{cor:itssquarefree} implies
\cite[Conjecture 4.2]{Conca}.
At present, we do not know how to prove Conca's stronger conjecture
to the effect that Corollary 
\ref{cor:itssquarefree} holds without the
hypothesis that the matrices $A_i$ are invertible~\cite[Conjecture~1.1]{Conca}.
One idea is 
to extend our study to multigraded Hilbert schemes on $K[X]$
whose defining Hilbert function is bounded above by~(\ref{eqn:ourHF}).

\begin{cor} \label{cor:connected}
The multigraded Hilbert scheme
$H_{d,n}$ is connected.
\end{cor}

\begin{proof}
All ideals in $H_{d,n}$ can be connected 
to their common generic initial ideal~$Z$
by Gr\"obner degenerations.
\end{proof}

In what follows we take a closer look at the combinatorics 
of the ideal~$Z$.

\begin{prop}
The ideal $Z$ is generated by all monomials
$x_{i_1 j_1} x_{i_2 j_2} .... x_{i_k j_k}$
where $1 \leq k - 1 \leq i_1,i_2,\ldots,i_k \leq d - 1$,
      $j_1 < j_2 < \cdots < j_k$,
and   $i_1 + i_2 + \cdots + i_k \leq d(k - 1)$.
The maximum degree of a minimal generator is
${\rm min}(d,n)$.
\end{prop}

\begin{proof}
This is the description of
the ideal $Z$ given by Conca
\cite[\S 5]{Conca}.
\end{proof}

All ideals $I$ in $H_{d,n}$ share the same Hilbert series 
in the ordinary $\Z$-grading,
\begin{equation*}\sum_{r=0}^\infty {\rm dim}_K (K[X]/I)_r \cdot z^r
\,\,\, = \,\,\,
 \frac{h(z)}{(1-z)^{n+d-1}}  .\end{equation*}
The {\em $h$-polynomial} in the numerator 
can be seen from the ideal of $2 {\times} 2$-minors:
\begin{equation*} h(z) \quad = \quad \sum_{i=0}^{{\rm min}(d-1,n-1)}
 \binom{d-1}{i} \cdot  \binom{n-1}{i} \cdot z^i . \end{equation*}
Note that $\,h(1) = \binom{n+d-2}{d-1} \,$ is the common
scalar degree of the ideals in $H_{d,n}$.

\begin{cor} \label{cor:CM}
Every ideal $I $ in $H_{d,n}$ is Cohen-Macaulay.
\end{cor}

\begin{proof}
Since $Z$ is the common generic initial ideal of 
all ideals $I$, it suffices to show that
the Borel-fixed ideal  $Z$ is Cohen-Macaulay.
Let $\Delta_Z$ denote the $(n+d-2)$-dimensional 
 simplicial complex corresponding to $Z$. The 
vertices of $\Delta_Z$ are the
$dn$ matrix entries $x_{ij}$,
and its facets are the $\binom{n+d-2}{d-1}$ 
sets $\, F_u \, = \,\{ x_{ij} : i > u_j \}\,$
which are complementary to the prime ideals $Z_u$.
We order the facets $F_u$ according to the 
lexicographic order on the vectors $u$.

We claim that this ordering of the facets is a 
{\em shelling} of $\Delta_Z$.  Since the Stanley-Reisner
ring of a shellable simplicial complex is 
Cohen-Macaulay \cite[Theorem~III.2.5]{Stanley}, this will imply 
Corollary \ref{cor:CM}. To verify the shelling property
we must show that every facet $F_u$ has a unique
subset $\eta_u$ such that the faces of $F_u$ 
not containing $\eta_u$ are exactly those appearing
as a face of an earlier facet.
If this condition holds then the $h$-polynomial
can be read off from the shelling as follows:
\begin{equation*}
h(z) \,\,\, = \,\,\, 
\sum_{u \in U} z^{\# \eta_u}.
\end{equation*}
The unique subset of the facet $F_u$ 
with these desired properties equals
\begin{equation*}
\eta_u \,\, = \,\,
\{\, x_{ij} \,:\,  j > 1 \,\,\hbox{and} \,\ i = u_j+1 < d \,\}. \end{equation*}
Indeed, suppose $G$ is a face common to $F_u$ and $F_{u'}$ for some $u' <
u$. Then $u_j' > u_j$ for some $j > 1$, so $G$ does not contain $x_{u_j+1,j}
\in \eta_u$. Conversely, suppose that $G$ 
is a face of $F_u$ which does not contain
$\eta_u$, and let $x_{u_j+1,j}$
be any element of $\eta_u \backslash G$. 
Since $j > 1$,
\begin{equation*}
F_{u+e_j-e_1} \, = \, F_u \backslash \{x_{u_j+1,j}\}
\cup \{ x_{u_1,1} \} \end{equation*}
is a facet of $\Delta_Z$ which contains $G$
and which comes earlier in our ordering.
\end{proof}

\begin{remark}
The shellability of $\Delta_Z$ was mentioned
in \cite[Remark 5.12]{Conca} but no details
were given there. It would be interesting to
know whether the simplicial complex $\Delta_I$ of 
every monomial ideal $I$ in $H_{d,n}$ is shellable.
\end{remark}

\section{Group completions}

In this section we relate our multigraded Hilbert scheme
to classical constructions in algebraic geometry.
For $n=2$ we recover the space of complete collineations
and its GIT construction due to Thaddeus in~\cite{Th}.
Brion~\cite{Brion} extended Thaddeus' work
to the diagonal $X \hookrightarrow X \times X$ of any
rational projective homogeneous variety
$X$.
While the present study is restricted to the case
$X = \PP^{d-1}$, we believe that many of our results will extend
to $X \hookrightarrow X^n$ in Brion's setting.

\begin{prop}
\label{prop:grothen}
There is a injective morphism from the multigraded Hilbert scheme~$H_{d,n}$ to a
connected component of the Grothendieck Hilbert scheme of subschemes
of~$(\PP^{d-1})^n$.
\end{prop}

\begin{proof}
By~\cite[\S 4]{HS}, there is a natural morphism from the multigraded Hilbert
scheme $H_{d,n}$ to the Grothendieck Hilbert scheme.
Theorem~\ref{thm:radical} shows that every point in $H_{d,n}$ corresponds to a
radical ideal~$I \subset K[X]$. Furthermore, the Hilbert function tells us that
the ideal $I + \langle x_{1j}, \ldots, x_{dj} \rangle$ has (affine) dimension $d+n-2$,
but by Corollary~\ref{cor:CM}, $I$ is pure of dimension $d+n-1$, so no
associated prime of $I$ contains $\langle x_{1j}, \ldots, x_{dj}\rangle$. Thus,
$I$ is uniquely determined by the subscheme it defines in $(\PP^{d-1})^n$, so
the morphism is injective.
\end{proof}

\begin{remark}
We do not know whether the morphism in Proposition~\ref{prop:grothen} is an
immersion, nor whether it is always surjective.
\end{remark}

Recall that the $G^n$-action on $H_{d,n}$ transforms ideals as  follows:
\begin{equation}
\label{eqn:action}
 I \,\,\mapsto \,\, (A_1,A_2,\cdots,A_n) \circ I .
\end{equation}
If $A_1 = A_2 = \cdots = A_n$ then the ideal $I$ is  left invariant,
so the stabilizer of any point $I \in H_{d,n}$ contains the
diagonal subgroup $\,G \,=\, \{(A,A,\ldots,A) \}\,$ of $G^n$.
Moreover, the stabilizer of the determinantal ideal $I_2(X)$
is precisely the diagonal subgroup $G$.
We write $\,\overline{G^n/G}\,$ for the closure
of $G^{n} \circ I_2(X)$ in the Hilbert scheme
$H_{d,n}$.

\begin{thm}
The subscheme $\overline{G^{n}/G}$ is an irreducible
component of $H_{d,n}$. It is a compactification of the homogeneous space
$G^{n}/G$, so it has dimension $ (d^2-1)(n-1)$.
\end{thm}

This theorem can be deduced from Proposition \ref{prop:grothen}
using standard algebraic geometry arguments concerning the
tangent sheaf of Grothendieck's Hilbert scheme. What we present below
is a more detailed combinatorial proof based on  the identification
of an explicit smooth point in Lemma  \ref{lem:diag}.

\begin{proof}
Clearly, the dimension of the tangent space at $I_2(X)$ is
at least $(d^2-1)(n-1)$, the dimension of $G^{n}/G$.
By semi-continuity, it is bounded above by the tangent 
space dimension of $H_{d,n}$ at any initial monomial ideal
${\rm in}(I_2(X))$. In Lemma \ref{lem:diag} below, we identify a particular
initial ideal for which this dimension equals $(d^2-1)(n-1)$.
From this we conclude that the tangent space of $H_{d,n}$ at $I_2(X)$
has dimension $\,(d^2-1)(n-1)$. This is precisely the dimension of
the orbit closure $\overline{G^{n}/G}$ of $I_2(X)$. We also conclude
that $I_2(X)$ is a smooth point of $H_{d,n}$, and that the 
unique irreducible component of
$H_{d,n}$ containing $I_2(X)$ is the compactified space $\overline{G^{n}/G}$.
\end{proof}

Let $M$ denote the ideal generated by the
quadratic monomials $x_{ik} x_{jl}$ for all
$1 \leq i < j \leq d$ and $1 \leq k < l \leq n$.
We call $M$ the {\em chain ideal}, because
its irreducible components correspond to chains
in the grid from $x_{1n}$ to $x_{d1}$. It is
a point  in  $\,\overline{G^{n}/G} \subset H_{d,n}\,$ 
since $M = {\rm in}(I_2(X))$ in the lexicographic order.

\begin{lem} \label{lem:diag}
The tangent space of the multigraded Hilbert scheme 
$H_{d,n}$ at the chain ideal $M$ has dimension  $\,(d^2-1)(n-1)$.
\end{lem}

\begin{proof}
We claim that the following three classes
$\rho$,  $\sigma$ and~$\tau$ form a basis for the tangent space
${\rm Hom}_{K[X]} (M,K[X]/M)_0 $:

\smallskip

\noindent {\em Class $\rho$}:
For each triple of indices $(i,j,l)$ 
with $1 \leq i < j \leq d$ and $1 < l \leq n$ 
we define a $K[X]$-module homomorphism
 $\,\rho_{ijl}\colon M \rightarrow K[X]/M\,$ by setting
\begin{align*}
\rho_{ijl} ( x_{hk} x_{jl}) &= x_{hk} x_{il} \quad
\hbox{whenever} \,\, i \leq h < j \,\,\hbox{and} \,\,k < l\hbox{, and} \\
\rho_{ijl}(m) &= 0 \quad
\hbox{for all other minimal generators $m$ of $M$}.
\end{align*}

\smallskip

\noindent {\em Class $\sigma$}:
For each triple of indices $(i,j,k)$ 
with $1 \leq i < j \leq d$ and $1 \leq k < n$ 
we define a $K[X]$-module homomorphism
 $\,\sigma_{ijk}\colon M \rightarrow K[X]/M\,$ by setting
\begin{align*}
\sigma_{ijk} ( x_{ik} x_{hl})  &=  x_{jk} x_{hl} \quad
\hbox{whenever} \,\, i < h \leq j \,\,\hbox{and} \,\,k < l\hbox{, and} \\
\sigma_{ijl}(m) &= 0  \quad
\hbox{for all other minimal generators $m$ of $M$}.
\end{align*}

\noindent {\em Class $\tau$}:
For each pair of indices $(i,k)$ 
with $1 \leq i < d$ and $1 \leq k < n$ 
we define a $K[X]$-module homomorphism
 $\,\tau_{ik}\colon M \rightarrow K[X]/M\,$ by setting
\begin{align*}
\tau_{ik} ( x_{i,k} x_{i+1,k+1}) &= x_{i,k+1} x_{i+1,k} \quad \hbox{and,}\\
\tau_{ik}(m) &=  0 \quad
\hbox{for all other minimal generators $m$ of $M$}.
\end{align*}

The above $K[X]$-linear maps are
$\Z^d$-homogeneous of degree zero, and
they are clearly  linearly independent over $K$.
There are $(n-1)(d-1)$ maps in the class 
$\tau$, and there are $(n-1)\binom{d}{2}$ each in
the classes $\rho$ and $\sigma$.
This adds up to the required total number of
$\,(d^2-1)(n-1) =  (n-1)(d-1)+2(n-1)\binom{d}{2}$.

It remains to be seen that every $\Z^d$-graded 
$K[X]$-module homomorphism
from~$M$  to $K[X]/M$ of degree zero is a $K$-linear
combination of the above. Suppose that
$\phi\colon M \rightarrow K[X]/M$ is a module homomorphism. Then,
 for $i < j$ and
$k < l$, we can uniquely write $\phi(x_{ik} x_{jl})$ as a linear combination
of monomials not in~$M$. Furthermore, by subtracting appropriate multiples of
$\rho_{ijl}$ and $\sigma_{ijk}$, we can assume that the monomials in
the linear combination do
not include
$x_{ik}x_{il}$ or $x_{jk}x_{jl}$. Suppose that for some $n \leq m$, the coefficient of $x_{mk} x_{nl}$ 
is some non-zero $\alpha\in K$. Either $i < m$ or $n < j$, and the two cases are
symmetric under reversing the order of both the column indices and the row indices, so we
assume the former.
For any $o$ such that $n \leq o \leq m$ and $i < o$, the syzygies imply
\begin{equation*}
\alpha x_{o,k+1}x_{mk}x_{nl}\,\, + \cdots
=\,\, x_{o,k+1} \phi(x_{ik} x_{jl}) 
\,\,=\,\, x_{jl} \phi(x_{ik} x_{o, k+1}).
\end{equation*}
Since the first term is non-zero in $K[X]/M$, the monomial must be divisible by
$x_{jl}$. Thus, either $j = n$, or both $j=o$
and $l = k+1$. In the first case, taking $o = m$, and using the assumption that
the coefficient of $x_{mk}x_{m,k+1}$ in $\phi(x_{ik}x_{o,k+1})$ is zero, we get
a contradiction.
In the second case, if $j \neq n$, then we must only
have one
choice of $o$ and this forces $i = n = m-1$. Therefore, $\phi$ is a linear
combination of homomorphisms of class $\tau$, and thus the classes of $\rho$,
$\sigma$, and~$\tau$ span the tangent space at $M$.
\end{proof}

\begin{cor}
\label{cor:chain-smooth}
The chain ideal $M$ is a smooth point on $H_{d,n}$.
The unique irreducible component of $H_{d,n}$
containing $M$ is the completion $\overline{G^{n}/G}$.
\end{cor}

We now turn to the case $n = 2$ which is well-studied in the
literature. The compactification $\overline{G^2/G}$ 
is the classical {\em space of complete collineations}, which was
investigated by Thaddeus in \cite{Th}. In fact, we have:

\begin{cor}
The multigraded Hilbert scheme $H_{d,2}$ is smooth and irreducible.
It coincides with the space of complete collineations:
 $\,H_{d,2} = \overline{G^2/G}$.
\end{cor}

\begin{proof}
Up to relabeling, the chain ideal is the
only monomial ideal in $H_{d,2}$. This point is
smooth by Lemma~\ref{cor:chain-smooth}, and hence $H_{d,2}$ is smooth.
Since it is connected by Corollary~\ref{cor:connected},
we conclude that $H_{d,2}$ is also irreducible.
The results in \cite{Th} show that the Grothendieck Hilbert scheme is isomorphic
to the space of complete collineations, and in particular smooth and
irreducible.
Thus, the morphism in Proposition~\ref{prop:grothen} is an isomorphism between
$H_{d,2}$ and the space of complete collineations.
\end{proof}

The representation of $\overline{G^2/G}$ as a multigraded Hilbert scheme
$H_{d,2}$ 
gives rise to nice polynomial
equations for the space of complete collineations.
 Namely, each ideal $I$ in  $ H_{d,2}$ is generated by $\binom{d}{2}$
equations of degree $(1,1)$.
As there are $d^2$ monomials in $K[X]_{(1,1)}$,
this describes an embedding of $H_{d,2}$ into
the Grassmannian ${\rm Gr}\big(\binom{d}{2},d^2\big)$.
The subscheme $H_{d,2}$ of this Grassmannian is cut out
by the determinantal equations which are derived
by requiring that the ideal  $I$  has the correct number
of first syzygies in degrees $(1,2)$ and $(2,1)$.

\begin{ex} [Equations defining $H_{3,2}$] \rm
We shall realize the $8$-dimensional manifold $H_{3,2}$ as a
closed subscheme of the  $18$-dimensional Grassmannian
${\rm Gr}(3,9)$, by giving explicit equations 
in the $84$ Pl\"ucker coordinates. Our equations 
furnish an explicit projective embedding for
Thaddeus' GIT construction \cite{Th} which is
reviewed further below.
Fix a $3 \times 9$-matrix of unknowns
\begin{equation*}
A \quad = \quad \begin{bmatrix}
a_{11} &a_{12} & a_{13} & a_{21} & a_{22} & a_{23} & a_{31} & a_{32} & a_{33}\\
b_{11} &b_{12} & b_{13} & b_{21} & b_{22} & b_{23} & b_{31} & b_{32} & b_{33}\\
c_{11} &c_{12} & c_{13} & c_{21} & c_{22} & c_{23} & c_{31} & c_{32} & c_{33}
\end{bmatrix}.
\end{equation*}
Consider the ideal $I$ generated  by 
the three bilinear polynomials in the vector
\begin{equation*}  A \cdot 
\bigl(
x_{11} x_{12},
x_{11} x_{22},
x_{11} x_{32},
x_{21} x_{12},
x_{21} x_{22},
x_{21} x_{32},
x_{31} x_{12},
x_{31} x_{22},
x_{31} x_{32} \bigr)^T. \end{equation*}
The condition for $I$ to be a point in $H_{3,2}$ is equivalent to
the condition that the rows of the following two $9{\times} 18$-matrices
are linearly dependent:
\begin{equation*} \!\!
\left[\begin{array}{cccccccccccccccccc}
a_{11} \!&\! a_{12} \!&\! a_{13} \!&\! a_{21} \!&\! a_{22} \!&\! a_{23} \!&\! a_{31} \!&\! a_{32} 
\!&\! a_{33} \!&\!  0 \!&\!  0 \!&\!  0 \!&\!  0 \!&\!  0 \!&\!  0 \!&\!  0 \!&\!  0 \!&\! 0 \! \\ \!
b_{11} \!&\! b_{12} \!&\! b_{13} \!&\! b_{21} \!&\! b_{22} \!&\! b_{23} \!&\! b_{31} \!&\! b_{32}
 \!&\! b_{33} \!&\!  0 \!&\!  0 \!&\!  0 \!&\!  0 \!&\!  0 \!&\!  0 \!&\!  0 \!&\!  0 \!&\! 0 \! \\ \!
c_{11} \!&\! c_{12} \!&\! c_{13} \!&\! c_{21} \!&\! c_{22} \!&\! c_{23} \!&\! c_{31} \!&\! c_{32} 
\!&\! c_{33} \!&\!  0 \!&\!  0 \!&\!  0 \!&\!  0 \!&\!  0 \!&\!  0 \!&\!  0 \!&\!  0 \!&\! 0 \! \\ \!
 0 \!&\!  0 \!&\!  0 \!&\! a_{11} \!&\! a_{12} \!&\! a_{13} \!&\!  0 \!&\!  0 \!&\!  0 \!&\! a_{21} 
\!&\! a_{22} \!&\! a_{23} \!&\! a_{31} \!&\! a_{32} \!&\! a_{33} \!&\!  0 \!&\!  0 \!&\! 0 \! \\ \!
 0 \!&\!  0 \!&\!  0 \!&\! b_{11} \!&\! b_{12} \!&\! b_{13} \!&\!  0 \!&\!  0 \!&\!  0 \!&\! b_{21} 
\!&\! b_{22} \!&\! b_{23} \!&\! b_{31} \!&\! b_{32} \!&\! b_{33} \!&\!  0 \!&\!  0 \!&\! 0 \! \\ \!
 0 \!&\!  0 \!&\!  0 \!&\! c_{11} \!&\! c_{12} \!&\! c_{13} \!&\!  0 \!&\!  0 \!&\!  0 \!&\! c_{21} 
\!&\! c_{22} \!&\! c_{23} \!&\! c_{31} \!&\! c_{32} \!&\! c_{33} \!&\!  0 \!&\!  0 \!&\! 0 \! \\ \!
 0 \!&\!  0 \!&\!  0 \!&\!  0 \!&\!  0 \!&\!  0 \!&\! a_{11} \!&\! a_{12} \!&\! a_{13} \!&\!  0 \!&\!  0
 \!&\!  0 \!&\! a_{21} \!&\! a_{22} \!&\! a_{23} \!&\! a_{31} \!&\! a_{32} \!&\! a_{33} \! \\ \!
 0 \!&\!  0 \!&\!  0 \!&\!  0 \!&\!  0 \!&\!  0 \!&\! b_{11} \!&\! b_{12} \!&\! b_{13} \!&\!  0 \!&\!  0 
\!&\!  0 \!&\! b_{21} \!&\! b_{22} \!&\! b_{23} \!&\! b_{31} \!&\! b_{32} \!&\! b_{33} \! \\ \!
 0 \!&\!  0 \!&\!  0 \!&\!  0 \!&\!  0 \!&\!  0 \!&\! c_{11} \!&\! c_{12} \!&\! c_{13} \!&\!  0 \!&\!  0 
\!&\!  0 \!&\! c_{21} \!&\! c_{22} \!&\! c_{23} \!&\! c_{31} \!&\! c_{32} \!&\! c_{33} 
\end{array}\right]
\end{equation*}
\begin{equation*} \!\!
\left[\begin{array}{cccccccccccccccccc}
a_{11} \!&\! a_{21} \!&\! a_{31} \!&\! a_{12} \!&\! a_{22} \!&\! a_{32} \!&\! a_{13} \!&\! a_{23} \!&\! a_{33} \!&\!  0 \!&\!  0 \!&\!  0 \!&\!  0 \!&\!  0 \!&\!  0 \!&\!  0 \!&\!  0 \!&\! 0 \! \\ \!
b_{11} \!&\! b_{21} \!&\! b_{31} \!&\! b_{12} \!&\! b_{22} \!&\! b_{32} \!&\! b_{13} \!&\! b_{23} \!&\! b_{33} \!&\!  0 \!&\!  0 \!&\!  0 \!&\!  0 \!&\!  0 \!&\!  0 \!&\!  0 \!&\!  0 \!&\! 0 \! \\ \!
c_{11} \!&\! c_{21} \!&\! c_{31} \!&\! c_{12} \!&\! c_{22} \!&\! c_{32} \!&\! c_{13} \!&\! c_{23} \!&\! c_{33} \!&\!  0 \!&\!  0 \!&\!  0 \!&\!  0 \!&\!  0 \!&\!  0 \!&\!  0 \!&\!  0 \!&\! 0 \! \\ \!
 0 \!&\!  0 \!&\!  0 \!&\! a_{11} \!&\! a_{21} \!&\! a_{31} \!&\!  0 \!&\!  0 \!&\!  0 \!&\! a_{12} \!&\! a_{22} \!&\! a_{32} \!&\! a_{13} \!&\! a_{23} \!&\! a_{33} \!&\!  0 \!&\!  0 \!&\! 0 \! \\ \!
 0 \!&\!  0 \!&\!  0 \!&\! b_{11} \!&\! b_{21} \!&\! b_{31} \!&\!  0 \!&\!  0 \!&\!  0 \!&\! b_{12} \!&\! b_{22} \!&\! b_{32} \!&\! b_{13} \!&\! b_{23} \!&\! b_{33} \!&\!  0 \!&\!  0 \!&\! 0 \! \\ \!
 0 \!&\!  0 \!&\!  0 \!&\! c_{11} \!&\! c_{21} \!&\! c_{31} \!&\!  0 \!&\!  0 \!&\!  0 \!&\! c_{12} \!&\! c_{22} \!&\! c_{32} \!&\! c_{13} \!&\! c_{23} \!&\! c_{33} \!&\!  0 \!&\!  0 \!&\! 0 \! \\ \!
 0 \!&\!  0 \!&\!  0 \!&\!  0 \!&\!  0 \!&\!  0 \!&\! a_{11} \!&\! a_{21} \!&\! a_{31} \!&\!  0 \!&\!  0 \!&\!  0 \!&\! a_{12} \!&\! a_{22} \!&\! a_{32} \!&\! a_{13} \!&\! a_{23} \!&\! a_{33} \! \\ \!
 0 \!&\!  0 \!&\!  0 \!&\!  0 \!&\!  0 \!&\!  0 \!&\! b_{11} \!&\! b_{21} \!&\! b_{31} \!&\!  0 \!&\!  0 \!&\!  0 \!&\! b_{12} \!&\! b_{22} \!&\! b_{32} \!&\! b_{13} \!&\! b_{23} \!&\! b_{33} \! \\ \!
 0 \!&\!  0 \!&\!  0 \!&\!  0 \!&\!  0 \!&\!  0 \!&\! c_{11} \!&\! c_{21} \!&\! c_{31} \!&\!  0 \!&\!  0 \!&\!  0 \!&\! c_{12} \!&\! c_{22} \!&\! c_{32} \!&\! c_{13} \!&\! c_{23} \!&\! c_{33} \! 
\end{array}\right]
\end{equation*}
These two matrices are obtained by multiplying the generators of $I$ 
with the 
entries in the two columns of $X = (x_{ij}) $ respectively. This results in
nine polynomials of bidegree $(2,1)$ and nine polynomials of bidegree $(1,2)$,
each having $18$ terms.
These two sets of polynomials must be linearly
dependent because each $I \in H_{3,2}$ 
has its first syzygies in these two bidegrees.

The $84$ maximal minors of the matrix $A$ are the
Pl\"ucker coordinates $p_{i_1 i_2, j_1 j_2, k_1 k_2}$
on the Grassmannian ${\rm Gr}(3,9)$, where
the indices run from~$1$ to~$3$.
Using Laplace expansion, we write each $9 {\times} 9$-minor of the 
two matrices as a cubic polynomial in these Pl\"ucker coordinates.
The condition that the matrices have rank at most eight translates into a
system of homogeneous cubic polynomials in the $84$ unknowns
$p_{i_1 i_2,j_1 j_2, k_1 k_2}$, and these
cubics define the space of complete collineations,
$\overline{G^2/G} = H_{3,2}$, as a subscheme of ${\rm Gr}(3,9)$.

Thaddeus \cite{Th} realizes 
$H_{3,2}$ as the (Chow or GIT) 
quotient of the Grassmannian ${\rm Gr}(3,6)$ by the one-dimensional subtorus of
$(K^*)^6$ given by the diagonal matrices with entries 
$(\,t,\,t,\,t,t^{-1},t^{-1},t^{-1})$.
We can see this in  our equations as follows. Let $U = (u_{ij})$ and
$V = (v_{ij})$ be $3 {\times} 3$-matrices of unknowns. Each point in 
${\rm Gr}(3,6)$ is 
represented as the row space of the $3 {\times} 6$-matrix $[U,V]$.
The group $G^2 = {\rm PGL}(3) \times {\rm PGL}(3)$ acts on 
$H_{3,2}$ by translating the
distinguished point 
$I_2(X)$ to the ideal generated by the three quadrics
\begin{equation*} \begin{matrix}
& (u_{i1} x_{11} + u_{i2} x_{21} + u_{i3} x_{31}) 
(v_{j1} x_{12} + v_{j2} x_{22} + v_{j3} x_{32}) \\
 - &
(u_{j1} x_{11} + u_{j2} x_{21} + u_{j3} x_{31}) 
(v_{i1} x_{12} + v_{i2} x_{22} + v_{i3} x_{32})
\end{matrix}
\qquad \hbox{for} \,\,1 \leq i < j \leq 3 . \end{equation*}
The entries of the corresponding $3 \times 9$ matrix $A$ are
\begin{equation*}
  a_{i_1 i_2} =  u_{1 i_1} \! v_{2 i_2} \!-\! u_{2 i_1} \! v_{1 i_2} ,\,\,
  b_{j_1 j_2} =  u_{1 j_1} \! v_{3 i_2} \!-\! u_{3 j_1} \! v_{1 j_2} ,\,\,
  c_{k_1 k_2} =  u_{2 k_1} \! v_{3 i_2} \!-\! u_{3 k_1} \! v_{2 k_2} .
\end{equation*}
Writing $u_{\mu}$ for the $\mu$-th column of the matrix $U$
and $v_{\nu}$ for the $\nu$-th column of $V$, this
translates into the following parametric representation
of $H_{3,2}$:
\begin{equation*}
p_{i_1 i_2, j_1 j_2, k_1 k_2} = 
  \det[u_{i_1},v_{i_2}, u_{j_1}]  \det[v_{j_2}, u_{k_1},v_{k_2}]
- \det[u_{i_1},v_{i_2}, v_{j_2}]  \det[u_{j_1}, u_{k_1},v_{k_2}].
\end{equation*}
These are quadratic polynomials in the Pl\"ucker coordinates on 
${\rm Gr}(3,6)$.
They are invariant under Thaddeus' torus action and antisymmetric
under swapping each of the three index pairs. Note that of the 
$84$ polynomials
only $12$ are actually Pl\"ucker binomials. Of the others, $6$ are zero 
(for example, $p_{11,21,31} = 0$) and $66$ are Pl\"ucker monomials 
(for example,
$p_{11, 21,32} = \det[u_1, v_1, u_2] \det[v_1, u_3, v_2]$).
This resulting map ${\rm Gr}(3,6) \rightarrow {\rm Gr}(3,9)$ 
gives an embedding of Thaddeus' quotient  $H_{3,2} = {\rm Gr}(3,6)/K^*$.
The cubic relations on ${\rm Gr}(3,9)$ described above
characterize the image of this embedding. \qed
\end{ex}

\begin{remark}
\label{remark:Lafforgue}
In the introduction of \cite{Laf},
Lafforgue describes the following compactification of $G^{n}/G$.
We consider the $d \times d$-minors of the
 $d \times (dn)$-matrix $\,(A_1,A_2,\ldots,A_n)$.
For each minor there is a corresponding vector ${\bf i}$ in
the set $\, D  =  \bigl\{
(i_1,i_2,\ldots,i_n) \in \N^n : i_1  + \cdots + i_n = d \bigr\} $,
namely, $i_j$ is the number of columns of $A_j$ occurring in that minor.
We introduce a new unknown $t_{\bf i}$ for each ${\bf i} \in D$,
and we multiply each minor by the corresponding unknown $t_{\bf i}$.
The scaled minors parametrize a subvariety in an affine space of dimension
$$ \binom{nd}{d} \,\,\, = \,\,\,
\sum_{{\bf i} \in D}  \binom{d}{i_1} \binom{d}{i_2}\cdots \binom{d}{i_n}. $$
This affine variety yields a projective variety $X_{d,n}$ which compactifies 
$G^{n}/G$:
\begin{equation}
\label{lafhook}
 G^{n}/G\,\,\, \hookrightarrow\,\,\,
 X_{d,n} \,\,\,\subset \,\,\,
 \prod_{{\bf i} \in D}\PP^{\binom{d}{i_1} \binom{d}{i_2}\cdots \binom{d}{i_n}-1}.
\end{equation}
In light of \cite[\S 2]{HS}, we can identify
$X_{d,n}$ with the partial multigraded Hilbert scheme
$(H_{d,n})_D$ obtained by restricting $H_{d,n}$ to the subset
of degrees $D \subset\Z^n$. Hence there is a natural 
morphism $H_{d,n} \rightarrow X_{d,n}$. This is an isomorphism
for $d = 2$ and $n=2$ but we do not know whether this is always the case.
In general, $X_{d,n}$ is singular, and the main result of \cite{Laf} is 
a combinatorial construction that replaces $X_{d,n}$
with another -- less singular -- model $\Omega_{d,n}$.
Yet, as discussed in the erratum to \cite{Laf},
$\,\Omega_{d,n}$ is not smooth for $d,n \geq 4$. \qed
\end{remark}

\section{Yet another space of trees}

This section concerns the case $d=2$.
The Hilbert scheme $H_{2,n}$
parametrizes degenerations of the projective line
in its diagonal embedding $\,\PP^1 \hookrightarrow (\PP^1)^n$.
Our goal is to prove the following two theorems about the structure of $H_{2,n}$.

\begin{thm} \label{thm:IrrButSing}
The multigraded Hilbert scheme $H_{2,n}$ is irreducible,
so it equals the compactification $ \overline{{\rm PGL}(2)^{n}/{\rm PGL(2)}}$.
However, $H_{2,n}$ is  singular for $n \geq 4$.
\end{thm}

Our second theorem explains why we
refer to $H_{2,n}$ as a {\em space of trees}. The qualifier
``yet another'' has been prepended 
to emphasize that this is not the 
 {\em space of phylogenetic trees}. The latter
is familiar to algebraic geometers as a discrete
model for $\overline{\mathcal{M}_{0,n}}$;
see \cite[Theorem~1.2]{GM} for a precise statement.

Following \cite{AS}, there is a natural graph structure 
on any multigraded Hilbert scheme, including $H_{2,n}$. The vertices are the
monomial ideals, and for every ideal in $H_{2,n}$ with precisely two initial
monomial ideals there is an edge between the corresponding vertices.
By \cite[Theorem 11]{AS}, this is precisely the induced subgraph 
on $H_{2,n}$ of the graph of all monomial ideals.
We note that our graph is not a {\em GKM graph}
in the sense of \cite{GHZ} because the $T^n$-fixed subvarieties
corresponding to edges usually have dimension greater than one.

\begin{thm} \label{thm:YASOT}
There are $\,2^n (n{+}1)^{n-2}\,$ monomial ideals in $H_{2,n}$,
one for each tree on $n{+}1$ unlabeled vertices
with $n$ labeled directed edges. Two trees are connected by
an edge on $H_{2,n}$ if they differ by one of the following operations:
\begin{enumerate}
\item Move any subset of the trees attached at a vertex to an adjacent vertex.
\item Swap two edges that meet at a bivalent vertex (preserving orientation).
\end{enumerate}
\end{thm}

In this section we use the following notation for our matrix of variables:
\begin{equation*}
X =  \begin{bmatrix}
x_1 & x_2 & \cdots & x_n \\
y_1 & y_2 & \cdots & y_n
\end{bmatrix}.
\end{equation*}
Thus $(x_i:y_i)$ are homogeneous coordinates
on the $i$-th factor in our
ambient space $(\PP^1)^n$.
The common Hilbert function (\ref{eqn:ourHF})
of all ideals $I$ in $H_{2,n}$ equals
\begin{equation} \label{eqn:ourHF2}
\N^n\, \rightarrow \,\N \,,\,\,\, (u_1,u_2,\ldots,u_n)\, \mapsto \, u_1 + u_2 + \cdots + u_n + 1 . 
\end{equation}
The unique Borel-fixed ideal in $H_{2,n}$ equals
$ \,  Z = \langle x_i x_j : 1 \leq i < j \leq n \rangle$.
Our first goal is to prove that $H_{2,n}$ is irreducible. This
requires a combinatorial description of the 
subvarieties $V(I)$ of $(\PP^1)^n$ 
corresponding to ideals $I \in H_{2,n}$.
Note that
each such subvariety is a reduced curve of multidegree $(1,1,\ldots,1)$ in
$(\PP^1)^n$.

\begin{lem} \label{lem:43}
The variety $V(I) \subset (\PP^1)^n $  defined by any ideal $I \in H_{2,n}$ is
the reduced union of several copies of
the projective line $\,\mathbb P^1$. For each factor of
$(\PP^1)^n$ there is exactly
one component of $V(I)$ which is not constant along this factor, and for this
component, the projection induces an isomorphism.
\end{lem}

\begin{proof}
Consider the projection from $V(I)$ onto the $i$-th factor of $(\PP^1)^n$.
We infer
from the Hilbert function (\ref{eqn:ourHF2}) that this projection is an isomorphism
over an open subset of $\PP^1$.
Hence there exists a rational map from $\mathbb P^1$ to a unique component $Y$ of
$V(I)$. The fact that the projection $Y \rightarrow \PP^1$ is a regular morphism implies that
the curve $Y$ is smooth. We conclude that the map $Y \rightarrow \PP^1$ is an  isomorphism.
\end{proof}

Each component of $V(I)$ can be labeled by the factors onto which it maps
isomorphically.  We draw $V(I)$
as a set of intersecting lines, labeled with subsets of the factors. 
By Lemma \ref{lem:43}, the labels form a partition of $\{1,2,\ldots,n\}$.
Moreover, since $K[X]/I$ is Cohen-Macaulay, $V(I)$ is connected,
and because only one component is non-constant along any factor, there is no
cycle among its components.
Hence our picture is an edge-labeled tree.
 We have the following converse to this description of the points of
$H_{2,n}$.

\begin{prop} \label{prop:construction-ideals-h2n}
Suppose that $Y \subset (\PP^1)^n$ is a
union of projective lines, which is connected and such that
each factor of $(\PP^1)^n$ has a unique 
projective line projecting isomorphically onto it.  Then
the radical ideal $I$ defining $Y$ is a point in $H_{2,n}$.
\end{prop}

\begin{proof}
We compute the Hilbert function and show that it coincides with (\ref{eqn:ourHF2}).
We proceed by induction on the number of components.
If $Y$ is irreducible then $Y$ is the translate of the diagonal
$\PP^1 $ in $(\PP^1)^n$ with some $A_i \in {\rm PGL}(2)$
acting on the $i$-th factor,
and therefore $\,I\, =\, (A_1,\ldots,A_n) \circ I_2(X)\,$ lies in $H_{2,n}$.

Now suppose $Y$ is reducible, let $Y_j$ be one of its components,
and  $F_j \subset \{1,\ldots,n\}$ the index set of factors of $(\PP^1)^n$
onto which $Y_j$ maps isomorphically. The prime ideal $I_j$ of $Y_j$
is generated by linear forms of multidegrees $\{e_i: i \not\in F_j \}$,
and by the $2 \times 2$-minors of a $2 \times |F_j|$-matrix $X_j$
which consists of the $F_j$ columns of $X$ acted on
by some $A_i \in {\rm PGL}(2)$.
The Hilbert function of $I_j$ is
\begin{equation*}
u \,\mapsto \,1 + \sum_{i \in F_j} u_i.
\end{equation*}
Since $Y$ is a tree of projective lines, there exists a component
$Y_j$ which has only one point of
intersection with the other components. Let 
$Y' \subset (\PP^1)^n$ be the union of the other components
and $I'$ the radical ideal defining $Y'$.  The ideal $I_j$ of
$Y_j$ contains linear forms of degree $e_i$ for every $i \not\in F_j$,
while $I'$ contains linear forms of degree $e_i$ for every $i \in F_j$.
This implies that $I' + I_j$ is a homogenous prime ideal 
generated by linear forms. Its variety equals $Y' \cap Y_j$,
and hence $I'+I_j$  has constant Hilbert function $1$. We conclude
\begin{align*}
\HF(I)  \,=\, \HF(I' \cap I_j) \,& =  \,\HF(I') + \HF(I_j) - \HF(I' + I_j)  \\
  &= \, \bigl( 1 + \sum_{i \in F_j} u_i \bigr) \,\, +\,\,
         \bigl( 1 +\sum_{i \not\in F_j} u_i \bigr) \,\,-\,\, 1,
\end{align*}
which  is the common Hilbert function (\ref{eqn:ourHF2})
of all ideals in $H_{2,n}$.
\end{proof}

Our discussion shows that each point in $H_{2,n}$
is characterized by the following data.
First, there is a tree of projective lines $Y_j = \PP^1$,  
labeled by the parts in a partition $\{1,\ldots,n\} = \cup_j F_j$.
These represent the factors of the ambient space $(\PP^1)^n$.
The intersection point  of two lines determines a marked
point on each of the two lines.
For each line labeled with more than one factor, we have a compatible set of
isomorphisms between those factors.
Given these data, we can compute the ideal $I_j$ of a component $Y_j$ as follows:
\begin{enumerate}
\item Let $X_j$ be the submatrix of $X$ given by the columns
indexed by $F_j$, acted on by the
$2 {\times} 2 $-matrices corresponding to the isomorphisms of~$\PP^1$.
\item For each $i \not\in F_j$ locate the 
intersection point on $Y_i$ that is nearest to $Y_j$.
Let $\alpha x_i + \beta y_i$ be the linear form defining this intersection point on $Y_i$.
 \item
The ideal $I_j$ is generated by these linear forms and the
$2 {\times} 2$ minors of~$X_j$.
The intersection ideal $I = \cap_j I_j$ is
the desired point in $H_{2,n}$.
\end{enumerate}

\begin{ex} \rm
The above algorithm implies that there are infinitely many
${\rm PGL}(2)^n$-orbits on $H_{2,n}$ when $n \geq 5$.
Consider a tree of four lines, $Y_1$, $Y_2$, $Y_3$, and~$Y_4$,
which meet a fifth line  in four distinct points, with
coordinates $(0{:}1),(1{:}1),(1{:}0)$ and $(t \!:\! 1)$ on $Y_5 = \PP^1$.
Each of these intersection points is identified with the
point $V(x_j) = \{(0 \! : \! 1)\}$ on the line $Y_j$.
Then we have
\begin{equation*}
\begin{matrix}
I_1 =  \langle x_2, x_3, x_4,x_5 \rangle &
I_2 = \langle  x_1,x_3,x_4,x_5 - y_5 \rangle & 
I_3 = \langle x_1, x_2, x_4, y_5 \rangle \\
I_4 = \langle  x_1, x_2, x_3, x_5 - t y_5 \rangle \! &
I_5 = \langle x_1, x_2, x_3, x_4 \rangle & \!
I = I_1 \cap I_2 \cap I_3 \cap I_4 \cap I_5
\end{matrix}
\end{equation*}
As $t$ varies over the field $K$, the ideals $I$ lie in different 
${\rm PGL}(2)^5$-orbits on $H_{2,5}$ because
the cross ratio of the four points on $Y_5$ 
is invariant under ${\rm PGL}(2)$. \qed
\end{ex}

\begin{proof}[Proof of Theorem \ref{thm:IrrButSing}]
We shall prove that $H_{2,n}$ is irreducible. Let $I$ be any ideal in $H_{2,n}$. 
We use induction on the number of components of $V(I)$ to show that
$I$ is in the closure of the orbit of $I_2(X)$.
If $V(I)$ is irreducible, then $I$ is in the orbit of $I_2(X)$ by the above
discussion, so we assume that
$V(I)$ has at least two components. We shall
construct another ideal $J \in H_{2,n}$
such that $V(J)$ has one fewer component than $V(I)$
and such that $J$ degenerates to $I$.

Consider the tree picture of $V(I)$ as described above, and
let $Y_1$ be a component
which has exactly one point
of intersection with the other components. Let $Y_2$ be one of these components
intersecting $Y_1$. After relabelling the factors and a change of coordinates,
we can assume that the isomorphisms of the factors associated to $Y_1$ are all
the identity map,
and the same holds for the isomorphisms of factors of $Y_2$. Furthermore,
we can assume that the point $Y_1 \cap Y_2$ is defined
in $(\PP^1)^n$ by the ideal
$\,\langle x_i : i \in F_1 \rangle +  \langle y_i : i \not\in F_1 \rangle $.

We now replace $Y_2$ with the component
$Y_2'$ labeled by the set $F_2' = F_1 \cup F_2$ (with the identity
isomorphisms), and we call its ideal $I'$. By
Proposition~\ref{prop:construction-ideals-h2n}, $I'$ is in $H_{2,n}$. The ideal $I_2'$ 
of $Y_2'$ is generated by  $\{ x_j : j \not\in F_2'\}$ and the $2 {\times} 2$
minors of the submatrix of $X$ indexed by $F_2'$.
 For $t \neq 0$, we
consider the ideal formed by replacing $y_j$ by $ty_j$ in $I'$ for $j \in F_1$
and take the flat limit as $t$ goes to $0$. The limit of $I_2'$ under this
action is $I_1 \cap I_2$, so the limit of $I'$ is contained in $I$. Since
$I$ and the limit of $I'$ have the same Hilbert function, they must be equal.
This proves the first assertion in Theorem~\ref{thm:IrrButSing}.
The second assertion will follow from Corollary~\ref{cor:trivalent-smooth} below.
\end{proof}

We now come to the combinatorial description of monomial ideals $I$
on $H_{2,n}$. Here the tree picture can be simplified.
There are precisely $n$ components $V(I) = Y_1 \cup Y_2 \cup 
\cdots \cup Y_n$, and the partition is into singletons $F_i = \{i\}$.
Each line $Y_i$ has only two points where intersections are possible,
namely,  $V(x_i) = \{(0\! : \!1)\}$ and $V(y_i) = \{(1\! : \! 0)\}$.
We draw $Y_i$ as an oriented line segment with the
intersection points only at the end points. 
The orientation is indicated by an arrow 
whose tail represents 
$V(x_i)$ and whose head represents  $V(y_i)$.
In this manner, each monomial ideal $I$ in $H_{2,n}$
is represented uniquely by a tree $T$ with
$n$ directed labeled edges. The tree $T$
has $n {+} 1$  vertices which remain unlabeled.
This establishes the first part of Theorem~\ref{thm:YASOT}.

Our construction is illustrated for $n=3$ in
Figure \ref{fig:claw-chain}. See Example \ref{example:twothree} 
below for a combinatorial discussion 
of the two  classes of trees shown here.

\begin{figure}
\begin{centering}
\includegraphics{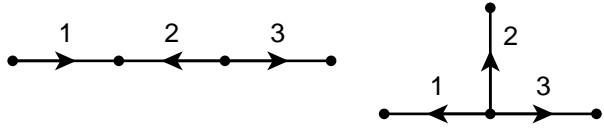}
\par\end{centering}
\caption{Edge-labelled trees corresponding to $2$ of the $32$ monomial ideals on
$H_{2,3}$. The tree on the left corresponds to the ideal $\langle y_1y_2,
y_1x_3, x_2x_3\rangle$ and on the right corresponds to the Borel-fixed ideal $Z
= \langle x_1x_2, x_1x_3, x_2x_3\rangle$.}
\label{fig:claw-chain}
\end{figure}

We next describe a rule for reading off the generators of
a monomial ideal $I \in H_{2,n}$ from its tree $T$.
For any two distinct indices $i$ and~$j$
in $\{1,\ldots,n\}$ we  
set $z_{ij} = x_j$ if the directed edge $j$ is pointing away from
the edge $i$ in $Y$ and $z_{ij} = y_j$ otherwise.
This means that the ideal $\,I_i \,$ of the component $Y_i$
 is generated by the variables
$\, z_{ij}\,$ for $j \in \{1,\ldots,n\}\backslash \{i\}$.
By intersecting these ideals for all $i$ we obtain the
following combinatorial formula for the ideal $I$.

\begin{remark}
The monomial ideal 
associated with the tree $T$ equals
\begin{equation*} I \,\, = \,\,\langle \,z_{ij}z_{ji} \,:\,
1 \leq i < j \leq n \,\rangle. \end{equation*}
Explicitly, the ideal generator corresponding to
pair $\{i,j\}$ of edges equals
\begin{equation*}
z_{ij} z_{ji} \,\,\, = \,\,\,
\begin{cases}
\,\, y_i y_j  & \hbox{ if the edges $i$ and $j$ point towards each other, } \\
\,\, x_i x_j & \hbox{ if the edges $i$ and $j$ point away from each other, }\\
\,\, y_i x_j & \hbox{ if the edge $i$ points to the edge $j$ but not conversely, } \\
\,\, x_i y_j & \hbox{ if the edge $j$ points to the edge $i$ but not conversely. }
\end{cases}
\end{equation*}
\end{remark}

Note that the Borel-fixed ideal $Z$ corresponds to 
 the star tree with all edges directed outwards.
Our next result concerns tangent spaces of $H_{2,n}$.

\begin{prop} \label{prop:tangentspace}
Let $I $ be the monomial ideal corresponding to a
directed tree $T$ as above.
Then the dimension of the tangent space of $H_{2,n}$ at $I$ is
\begin{equation} \label{eqn:fsum}
\sum_{v} f(\operatorname{degree}(v))
\end{equation}
where the sum is over all vertices of $T$, and the function~$f$ is defined
by
\begin{equation*}
f(a) = \begin{cases} \,3(a-1) & \hbox{if} \,\,\, 1 \leq a \leq 3, \\
\, a(a-1) & \hbox{if} \,\,\, a \geq 3. \end{cases}
\end{equation*}
\end{prop}

The following corollary to this result
completes the proof of Theorem~\ref{thm:IrrButSing}.

\begin{cor} \label{cor:trivalent-smooth}
The monomial ideal $I$ is a smooth point on the
Hilbert scheme $H_{2,n}$ 
if and only if every vertex in the tree $T$ is at most trivalent.
\end{cor}

\begin{proof}
The tree $T$ has $n+1$ vertices $v$,
and the number of edges is
\begin{equation*}
n  \,\, = \,\, \frac{1}{2} \cdot \sum_v \operatorname{degree}(v) 
\end{equation*}
Since $H_{2,n}$ is a compactification of ${\rm PGL}(2)^{n-1}$,
its dimension equals
\begin{equation*} {\rm dim}(H_{2,n}) \,\, = \,\,
3(n-1) \,\, = \,\,  6 n - 3(n+1) \,\, = \,\,  
\sum_{v} 3(\operatorname{degree}(v) - 1).
\end{equation*}
Since $f(a) \geq 3(a-1)$, with equality if and only if $a \leq 3$,
the sum in (\ref{eqn:fsum}) is equal to ${\rm dim}(H_{2,n})$,
if and only if
$\, {\rm degree}(v) \leq 3$ for all vertices $v$.
\end{proof}

\begin{example}
The star tree ideal
$Z = \langle \,x_i x_j \,:\, 1 \leq i < j \leq n \,\rangle\,$
has tangent space dimension $n(n-1)$.
In fact, it is the most singular point on $H_{2,n}$, since
every other ideal degenerates to $Z$.
On the other hand, the chain ideal
$\,M = \langle\, x_i y_j \,: \, 1 \leq i < j \leq n \, \rangle \,$
is a smooth point on $H_{2,n}$, because the tree for $M$ is a 
chain of $n$ directed edges. This confirms
Lemma \ref{lem:diag} for $d=2$. \qed
\end{example}

\begin{proof}[Proof of Proposition  \ref{prop:tangentspace}]
For any distinct edges $k$ and $\ell$ meeting at a vertex $v$
of the tree $T$, we define the
following tangent directions for $H_{2,n}$ at $I$:
\begin{equation*}
\alpha_{k \ell} \colon z_{ij}z_{ji} \mapsto
\begin{cases}
\, z_{ij} \tilde{z}_{ji} & \mbox{if } i = k \mbox{ and $j$ connects to $v$ via
$\ell$ (including $j=\ell$),} \\
\,\,\,\, 0 & \mbox{otherwise.}
\end{cases}
\end{equation*}
Here we use the convention that ${\tilde z}_{ij} = x_j$ if $z_{ij} = y_j$ and
${\tilde z}_{ij} = y_j$ if $z_{ij} = x_j$.
Moreover, if $v$ is bivalent, i.e.\ $k$ and $\ell$ are the only edges incident
to $v$, define:
\begin{equation*}
\beta_{v} \colon z_{ij} z_{ji} \mapsto \begin{cases}
\, \tilde{z}_{ij}  \tilde{z}_{ji} & \mbox{if } \{i,j\} = \{k, \ell\}, \\
\,\,\,\, 0 & \mbox{otherwise.}
\end{cases}
\end{equation*}
If the vertex $v$ has degree $a$, then we have defined $f(a)$ maps.
To show that these map are indeed tangent directions, we exhibit a one-parameter
deformation of $I$. The tangent vector $\beta_{v}$ is realized by the
curve on $H_{2,n}$ gotten by replacing $z_{k \ell} z_{\ell k}$ with 
$z_{k \ell} z_{\ell k} - \epsilon \tilde{z}_{k \ell}  \tilde{z}_{\ell k}$
among the generators of $I$. The resulting ideal in $H_{2,n}$
 represents the tree of lines gotten by merging the edges
$k$ and $\ell$ to a single $\PP^1$ labeled by $\{k,\ell\}$.
The tangent vector $\alpha_{k \ell}$ is realized by
replacing $z_{ji}$ with $z_{ji} - \epsilon \tilde{z}_{ji}$
in all prime components of $I_j$ such that
$j$ connects to $v$ via $\ell$. The resulting ideal in $H_{2,n}$
represents the tree of lines gotten by sliding the
subtree at $v$ in direction $\ell$ along the
edge labeled~$k$.

As $v$ ranges over all vertices of the tree $v$,
and $k, \ell$ range over all incident edges,
we now have a collection of tangent vectors
whose cardinality equals~(\ref{eqn:fsum}).
To see that these vectors are linearly independent, we note that
$\alpha_{k\ell}$ is the only one of these tangent vectors such that the image of
$z_{k\ell}z_{\ell k}$ has a non-zero coefficient for $z_{ij} \tilde z_{ji}$ and
that $\beta_v$ is the
only one such that the image of $z_{k\ell}z_{\ell k}$ has a non-zero
coefficient for $\tilde z_{ij} \tilde z_{ji}$. Thus, a non-trivial linear
combination of the $\alpha_{k\ell}$ and $\beta_v$ can't be the zero tangent
vector.

It remains to be seen that our tangent vectors
span the tangent space. Suppose there exists
a tangent vector $\phi$ that is not in the span of the
$\alpha_{k\ell}$ and $\beta_v$. After subtracting suitable multiples of these
known tangent vectors, we may assume that
for any pair of adjacent edges $i$ and $j$ 
there exists a scalar $\nu_{ij}$ such that
$\phi(z_{ij}z_{ji}) = \nu_{ij} \tilde z_{ij} \tilde z_{ji}$,
and furthermore $\nu_{ij} = 0$ if 
the node $v$ shared by $i$ and $j$ is bivalent.
Suppose that $v$ has degree at least $3$ and let
 $k$ be an edge incident to $v$ distinct
from $i$ and $j$. Then $z_{ij} = z_{kj}$ and hence
\begin{equation*}
\nu_{ij} z_{jk} \tilde z_{ij} \tilde z_{ji}
\, = \, z_{jk} \phi(z_{ij} z_{ji})
\, = \, \phi(z_{jk}z_{kj})z_{ji}
\, = \, \nu_{jk} \tilde z_{jk} \tilde z_{kj} z_{ji}.
\end{equation*}
This implies  $\nu_{ij} = \nu_{ik} = 0$.
We conclude that
$\phi(z_{ij}z_{ji}) = 0$ for
any pair of adjacent edges $i$ and $j$.
Now suppose that $i$ and $j$ are not adjacent and write
\begin{equation*}
\phi(z_{ij}z_{ji}) \,\, =  \,\, \lambda z_{ij} \tilde{z}_{ji}\, +\, \mu \tilde{z}_{ij} z_{ji}
\, + \, \nu \tilde{z}_{ij} \tilde{z}_{ji}.
\end{equation*}
Let $\ell$ be the edge adjacent to $j$ on the path from $i$ to $j$.
Then $z_{ij} = z_{\ell j}$ and
\begin{equation*} 0 \,=\, z_{ji} \phi(z_{\ell j} z_{j \ell}) \,= \,
\phi(z_{ij}z_{ji}) z_{j \ell} \\
\,=\, 0 + \mu  \tilde z_{ij} z_{ji} z_{j \ell} + \nu \tilde z_{ij} \tilde
z_{ji} z_{j \ell} \quad\, \hbox{modulo $I$.} \end{equation*}
This implies $\mu=\nu=0$ and, by symmetry, $\lambda=0$.
We conclude that $\phi=0$, so our maps $\alpha_{k \ell}$ and
$\beta_v$ form a basis for the tangent space of $H_{2,n}$ at $I$.
\end{proof}

\begin{proof}[Proof of Theorem \ref{thm:YASOT}]
We already saw that the monomial ideals on $H_{2,n}$
are in bijection with trees $T$ with $n{+}1$ unlabeled
vertices and directed edges that are labeled
with $\{1,\ldots,n\}$. To show that there are $2^n (n{+}1)^{n-2}$
monomial ideals, it suffices to show there are $(n{+}1)^{n-2}$
edge-labeled trees on $n{+}1$ vertices. Picking an arbitrary node as the root and
shifting the labels to the nodes away from the root gives a rooted,
node-labeled tree, of which there are $(n{+}1)^{n-1}$. From the rooted tree, we
can uniquely recover the edge-labeled tree and the choice of the
root, so there are $(n{+}1)^{n-2}$ edge-labeled trees.

We now need to identify
all ideals $I$ in $H_{2,n}$ that possess precisely
two initial monomial ideals.
We already saw two classes of such ideals in the proof of 
Proposition  \ref{prop:tangentspace}. First, there was the ideal with generator 
$z_{k \ell} z_{\ell k} - \epsilon \tilde{z}_{k \ell}  \tilde{z}_{\ell k}$
which realizes the deformation $\beta_v$ and swap \# 2 in the statement of 
Theorem \ref{thm:YASOT}. We also exhibited an ideal for the deformation $\alpha_{k \ell}$  
which realizes the move \# 1 when the subset of trees is a singleton.
The general case is subsumed by the following argument.

In light of Theorem \ref{thm:radical} and Proposition \ref{prop:grothen},
Gr\"obner degenerations of ideals $I$ in $H_{2,n}$
correspond  to scheme-theoretic limits of the trees
$Y$ with respect to one-parameter subgroups of the
$(K^*)^n$-action on $(\PP^1)^n$.

Let $Y$ be a tree on $H_{2,n}$ that has
precisely two degenerations to $(K^*)^n$-fixed trees.
There are two cases to be considered. First suppose that
 some component $Y_i$ of $Y$ is labeled by a 
subset $F_i \subseteq \{1,\ldots,n\}$  with $|F_i| \geq 2$.
The component $Y_i$ admits $|F_i|!$ distinct degenerations
to a $(K^*)^n$-fixed tree, and, by Proposition  \ref{prop:construction-ideals-h2n},
each of these lifts to a degeneration of $Y$. This implies that
the tree $Y$ has $n-1$ edges, and the unique non-singleton label
$F_i$ has cardinality two. Moreover, each intersection point
$Y_j \cap Y_k$ is a torus-fixed point on both $Y_j$ and $Y_k$.
This is precisely the situation in swap \# 2 above.

In the second case  to be considered, the 
tree $Y$ consists of the $n$ lines $Y_1,\ldots,Y_n$,
each labeled by a singleton. Consider all components $Y_i$ with
intersection points that are not torus-fixed.
For each such component $Y_i$ there exist two
torus degenerations of $Y$ that move the intersection 
points on $Y_i$ to the two torus-fixed points on $Y_i$.
Under these degenerations,
the intersection points $Y_j \cap Y_k$ with $i \not\in \{j,k\}$ 
remain in their positions on $Y_j$ and $Y_k$.
Hence, only one component
$Y_i$ has intersection points that are not torus-fixed.
This is precisely the situation in move \# 1 in 
Theorem \ref{thm:YASOT}.
\end{proof}

\begin{example} \label{example:twothree} ($n=3$) \
The Hilbert scheme $H_{2,3}$ has $32$ monomial ideals, corresponding to the 
eight
orientations on the claw tree and to the eight orientations on each of the
three labeled bivalent trees. Representatives for the two classes
of trees are shown in Figure \ref{fig:claw-chain}.
 The eight orientations of the claw tree can be
arranged into the vertices of a cube. Each edge in the cube is an edge in the
graph corresponding to moving two edges at a time between vertices. Along each
edge, add two vertices corresponding to bivalent trees and four edges from each
of these to each adjacent vertex of the cube. In addition to these operations
corresponding to move \# 1, there are are 24~edges corresponding to swap \# 2.
These are arranged into four
hexagons.

The six-dimensional manifold $H_{2,3}$ coincides 
with Lafforgue's compactification $X_{2,3}$ in Remark \ref{remark:Lafforgue}.
Here, (\ref{lafhook}) amounts to an embedding
of $\,H_{2,3} \,$ into $\,\PP^3 \times \PP^3 \times \PP^3$.
The equations for this embedding are as follows.
Each $\PP^3$ parametrizes one of the three generators
$\, a_{ij} x_i x_j + b_{ij} x_i y_j +  c_{ij} y_i x_j + d_{ij} y_i y_j $,
$1 \leq i < j \leq 3$, of an ideal in $H_{2,3}$.
Being a point in $H_{2,3}$ means that these ideal generators
 admit two linearly independent
syzygies in degree $(1,1,1)$. This happens if and only if
the following $6 \times 9$-matrix has rank at most four:
$$ \begin{pmatrix}
 a_{12} &  0 &  b_{12} &  0 &  c_{12} &  0 &  d_{12} &  0  \\
 0 &  a_{12} &  0 &  b_{12} &  0 &  c_{12} &  0 &  d_{12} \\
 a_{13} &  b_{13} &  0 &  0 &  c_{13} &  d_{13} &  0 &  0  \\
 0 &  0 &  a_{13} &  b_{13} &  0 &  0 &  c_{13} &  d_{13} \\
 a_{23} &  b_{23} &  c_{23} &  d_{23} &  0 &  0 &  0 &  0  \\
 0 &  0 &  0 &  0 &  a_{23} &  b_{23} &  c_{23} &  d_{23} 
\end{pmatrix}
$$
By saturating its ideal of $5 \times 5$-minors 
with respect to the irrelevant ideal
$\,\bigcap_{1 \leq i < j \leq 3}
\langle a_{ij}, b_{ij},  c_{ij}, d_{ij} \rangle $, we find
that the prime ideal of $X_{2,3} {\subset} (\PP^3)^3$ is
generated by nine cubics such as 
$  a_{12}  a_{13} d_{23}
- a_{12} b_{13} c_{23} 
- b_{12} a_{13} b_{23}
+ b_{12} b_{13} a_{23}$.
\qed
\end{example}

\section{Three projective planes} \label{sec:three-planes}

In this section we study the smallest case where $H_{d,n}$ is reducible,
namely, $n=d=3$.
The multigraded Hilbert scheme $H_{3,3}$
parametrizes degenerations of
the projective plane in its diagonal embedding
$\PP^2 \hookrightarrow \PP^2 \times \PP^2 \times \PP^2$.
We use the following notation for the
unknowns $x_{ij}$ in the polynomial ring $K[X]$.
\begin{equation*}
X \quad = \quad
\begin{bmatrix}
x_1 & x_2 & x_3 \\
y_1 & y_2 & y_3 \\
z_1 & z_2 & z_3
\end{bmatrix}
\end{equation*}

\begin{thm} \label{thm:three-planes}
The multigraded Hilbert scheme $H_{3,3}$ is the reduced union of
seven irreducible components,
each of which contains a dense 
${\rm PGL}(3)^3$ orbit:
\begin{itemize}
\item The $16$-dimensional main component
$\,\overline{{\rm PGL}(3)^3/{\rm PGL}(3)}\,$ is singular.
\item Three $14$-dimensional smooth components are permuted under the
$S_3$-action on $(\PP^2)^3$. At a generic point, the subscheme of
$(\PP^2)^3$ is the union of the blow-up of $\PP^2$ at a point, two copies of
$\PP^2$, and $\PP^1 \times \PP^1$. An ideal which represents
such a point on this component is
\begin{multline} \label{eqn:ideal-extra-14}
\langle \,x_1, x_2, y_1 z_2 - z_1 y_2, y_1y_3-z_1x_3, y_2y_3-z_2x_3
\, \rangle \,\,\cap \\
\langle x_1, y_1, x_3, y_3\rangle \cap \langle x_1, x_2, x_3, y_3 \rangle \cap
\langle x_2, y_2, x_3, y_3 \rangle.
\end{multline}
\item Three $13$-dimensional smooth components
are permuted under
the  $S_3$-action on $(\PP^2)^3$. A generic point looks like the union of three copies
of $\PP^2$
and $\PP^2$ blown up at three points.
A representative ideal is
\begin{equation*} \label{eqn:ideal-extra-13}
\langle x_1, y_1, x_2, z_2 \rangle \cap \langle x_1, y_1, x_3, y_3 \rangle
\cap \langle x_2, y_2, x_3, y_3 \rangle \cap
\langle x_1, x_2, x_3, y_1 y_2 z_3 - z_1 z_2 y_3 \rangle.
\end{equation*}
\end{itemize}
\end{thm}

With some additional notation, we can describe the isomorphism types of
the six extra components. Let $\Fl$ denote the variety of
complete flags in $K^3$ and $\mathcal O_i$ the tautological bundle of
$i$-dimensional vector spaces for $i=1$ or~$2$. Then the second class of
components are
isomorphic to the bundle
\begin{equation} \label{eqn:extra-14-bundle}
\mathcal H_{2,3} \big(\PP^2 \rightarrow \PP(\mathcal O_2)
\times \PP(\mathcal O_2)
\times \PP(\mathcal O^{\oplus 3} / \mathcal O(-1))\big)
\,\rightarrow \,\Fl \times \Fl \times \PP^2
\end{equation}
where $\mathcal H_{2,3}$ is a bundle whose fibers are each isomorphic to
the Hilbert scheme $H_{2,3}$.
A point in this bundle
is equivalent to a point $x$ in $\Fl \times \Fl \times \PP^2$,
together with an ideal in the total coordinate ring of $\PP((\mathcal O_2)_x)
\times \PP((\mathcal O_2)_x) \times \PP((\mathcal O^{\oplus 3}/\mathcal
O(-1))_x))$ with the appropriate Hilbert function~(\ref{eqn:ourHF}).

To relate this formulation to the ideal in~(\ref{eqn:ideal-extra-14}), we 
identify linear forms in $K[X]$ with the direct sum of three vector
spaces, each of dimension~$3$, and $\{x_i, y_i, z_i\}$ as choice of basis for
the $i$th summand.
The two flag varieties parametrize the duals of the flags $\langle x_i
\rangle \subset \langle x_i, y_i\rangle$ for $i = 1,2$. The projective space
parametrizes the point whose ideal is $\langle x_3, y_3\rangle$. These 
spaces determine all the linear generators in~(\ref{eqn:ideal-extra-14}). The
additional generators of the first component 
represent a point in $H_{2,3}$, but without a
canonical choice of basis. 

The third class of components are isomorphic to the projective bundle:
\begin{equation*}
\PP(\mathcal E) \,\rightarrow\, \Fl \times B \times \Fl
\end{equation*}
where $\Fl$ is as before and $B$ is the blow-up of $\PP^2 \times \PP^2$ along
the diagonal.
We think of the blow-up variety $B$
as the parameter space of two points in
$\PP^2$
and a line containing them. The $4$-dimensional vector bundle
$\mathcal E$ is the sum
\begin{equation} \label{eqn:bundle-13}
(\Osh_1 \otimes \Osh_1 \otimes \Osh_2) + (\Osh_1 \otimes \Osh_2
\otimes \Osh_1) + (\Osh_2 \otimes \Osh_1' \otimes \Osh_1)
\end{equation}
inside $\Osh_{\Fl}^{\oplus 3} \otimes \Osh_{B}^{\oplus 3} \otimes
\Osh_{\Fl}^{\oplus 3}$, where $\Osh_1$ and $\Osh_1'$ are the pullbacks to $B$ of
$\Osh(-1)$ on each of the copies of $\PP^2$, which parametrize the two points.
The flag varieties parametrize the duals of $\langle x_i
\rangle \subset \langle x_i, y_i\rangle$ for $i = 1, 3$ and $B$ parametrizes the
two points defined by $\langle x_2, y_2\rangle$ and $\langle x_2, z_2\rangle$,
with $\langle x_2 \rangle$ as the line between them.
As before, these vector spaces determine the linear generators of the
components.  The bundle $ \PP(\mathcal E)$
parametrizes the coefficients of the cubic
generator of the ideal of the blowup of $\PP^2$ at two points. This ideal
equals
\begin{equation}
\label{eqn:13-cubic}
\langle \,x_1, x_2, x_3, \,
a y_1 y_2 z_3  + b y_1 y_2 y_3 + c y_1 z_2 y_3 + d z_1 z_2 y_3 \,\rangle
\end{equation}
Note that the middle two terms are linearly independent even when $y_2$ and
$z_2$ coincide.
For generic coordinates, after a change of basis, we can take $b$ and
$c$ to be zero, and after rescaling, we take $a=d=1$. Thus, the
${\rm PGL}(3)^3$ orbit of the ideal~(\ref{eqn:ideal-extra-13}) is dense in the 
$13$-dimensional component of $H_{3,3}$.

\begin{proof}[Proof of Theorem~\ref{thm:three-planes}:]
The proof is computational. It rests on the
 {\tt Singular} code  posted at
\url{http://math.berkeley.edu/~dustin/diagonal/h33.sng}.
The computation works regardless
of what the field characteristic of $K$ is.

It suffices to consider an affine neighborhood
of the unique Borel-fixed ideal $Z$ in $H_{3,3}$. 
We employ the
standard method of chosing coordinates by adding
trailing terms with indeterminate coefficients
to the ten monomial generators of $Z$. The ideal defining $H_{3,3}$ is
then derived from the syzygies of $Z$.
We then derive the prime ideals representing each of the
seven components, by translating
the geometric descriptions above
into local coordinates around $Z$. Implicitization 
using {\tt Singular} yields the seven prime ideals, and we
check that their intersection equals the
ideal of the Hilbert scheme itself.

We now explain how the parametric representations of the seven components
are derived.
The main component is, by definition, parametrized by the 
${\rm PGL}(3)^3$ orbit of
the ideal of  $2\times 2$ minors of $X$.
The other components can also be parametrized by ${\rm PGL}(3)^3$
orbits of the representative ideals, but it is also possible 
-- and computationally more efficient --
 to use parametrizations which do not require localization.

For the $14$-dimensional components, we begin by using local coordinates in
$H_{2,3}$ to define a family of subschemes of $(\PP^1)^3$ over $\mathbb A^6$.
Renaming some
of the variables and adding two linear terms, we get the parametrization of the
blow-up of $\mathbb P^2$ and its degenerations in $(\mathbb P^2)^3$, i.e.\ a
neighborhood of a fiber of (\ref{eqn:extra-14-bundle}).
Intersecting with the ideals of linear spaces gives the family in $K[X]$ with
the appropriate Hilbert function. Different choices of
flags can be represented by upper triangular changes of coordinates on each of
the three columns of $X$.
The parametrization of the $13$-dimensional component follows the same 
pattern. In a neighborhood of $Z$, up to change of basis, we can take 
the two points parametrized by $B$ to be $\langle x_2, y_2 \rangle$ and 
$\langle x_2, y_2 - az_3 \rangle$ with $\langle x_2 \rangle$ as the line
between them. In addition to $a$, the other coordinates are the entries 
of the upper triangular matrix corresponding to the choice of flag and 
to the coefficients of the cubic generator in (\ref{eqn:13-cubic}).
Implicitizing these parametrizations reveals the prime ideals for these seven
components, and their intersection 
is found to equal the ideal
of the Hilbert scheme itself.
\end{proof}

\begin{table}
\begin{centering}
\begin{tabular}{|l|c|c|c|c|c|c|}
\hline
& & & \multicolumn{3}{c|}{component} & \\
 & T. sp. & planar? & main? & 14-dim.? & 13-dim.? & symm. \\
\hline
1 & 16 & y & y & n & n & 2 \\
2 & 16 & y & y & n & n & 1 \\
3 & 16 & y & y & n & n & 1 \\
4 & 18 & y & y & n & n & 6 \\
5 & 16 & y & y & n & n & 3 \\
6 & 14 & y & n & y & n & 2 \\
\hline
7 & 15 & y & n & y & y & 1 \\
8 & 16 & n & y & n & n & 1 \\
9 & 17 & n & y & y & n & 1 \\
10 & 18 & n & y & n & n & 2 \\
11 & 17 & n & y & y & n & 1 \\
12 & 14 & n & n & y & n & 2 \\
\hline
13 & 18 & n & y & y & y & 2 \\
14 & 18 & n & y & y & n & 2 \\
15 & 18 & n & y & y & y & 1 \\
\hline
16 & 18 & n & y & y & y & 6 \\
\hline
\end{tabular}
\par
\end{centering}
\caption{The symmetry classes of monomial ideals in $H_{3,3}$}
\label{tbl:monomials}
\end{table}

In Table~\ref{tbl:monomials} we show that 
$H_{3,3}$ contains $13824$ monomial ideals.
They come in $16$ symmetry classes, and we list them
in four groups, corresponding to the dimension
($12$, $11$, $10$, and~$9$) of the orbit under the action of
${\rm PGL}(3)^3$. The $16$ monomial ideals appear in the same order 
as their pictorial representation in Figure~\ref{fig:monomial-poset}.
The third column indicates whether or not the picture is planar.
The second column is the tangent space dimension of $H_{3,3}$ at that
point. The triple column shows which  components the monomial
ideals live on. The rightmost column shows the order of the
symmetry group of the ideal. Note that the permutation group acting
on $H_{3,3}$ has order $6^4 = 1296$: it permutes the three factors
of $\PP^2 {\times} \PP^2 {\times} \PP^2$ as well as the three
coordinates $\{x_i,y_i,z_i\}$ of each projective plane.
The total number $13824$ of monomial ideals   on
$H_{3,3}$ equals $1296$ times the sum of the reciprocals in the last
column of  Table~\ref{tbl:monomials}.

Every monomial ideal in $H_{3,3}$ corresponds to a polyhedral
complex in the boundary in the direct product
of three triangles, denoted $(\Delta_2)^3$. Using the moment
map of toric geometry, each such polyhedral complex can be 
identified with the real positive points of the 
corresponding subscheme of $(\mathbb
P^2)^3$.  The following conditions characterize those subcomplexes of $H_{3,3}$
whose corresponding monomial ideal has the right multigraded Hilbert function:

\begin{itemize}
\item For every vector of non-negative
integers $(t_1, t_2, t_3)$ summing to $2$, there is exactly one $2$-dimensional
cell consisting of the product of a $t_1$-dimensional cell of $\Delta_2$, a
$t_2$-dimensional cell, and a $t_3$-dimensional cell. 
\item The complex contains exactly ten $0$-cells, $15$ $1$-cells,
and six $2$-cells.
\end{itemize}
This characterization is sufficient to show that Table~\ref{tbl:monomials} is
complete. By the condition on the number of $1$-cells, each triangle must meet
at least two of the squares. Thus, each pair of squares must be adjacent or be
connected by a triangle, and so all three squares must meet in a common point.
There are four possible configurations of the squares: either $0$, $1$, $2$,
or~$3$ edges common to multiple squares. The monomial ideals can be enumerated
by considering all possible ways to attach the additional triangles to these
configurations.

The monomial ideals form a poset based on containment within the closures of
their orbits, illustrated in Figure~\ref{fig:monomial-poset}. 
 Each ideal is drawn as a $2$-dimensional subcomplex of
$(\Delta^2)^3$. The subcomplex is drawn abstractly,
but the embedding amounts to a choice of labellings. The bold lines indicate
that an additional triangle is attached along that edge.
By orbits, we mean orbits under the disconnected group which is
generated by multiplying the first column by an arbitrary matrix and by the
discrete action of permuting the columns.  The number on the lower
right is the dimension of the tangent space of $H_{3,3}$ at that
monomial ideal. The ranking is by the dimension of (every component of) the
orbit of the monomial ideal: 9, 10, 11, or~12.

The maximal elements of the poset in Figure~\ref{fig:monomial-poset}
 correspond to the ``planar''
complexes, i.e.\ those such that no edge contains more than two 2-cells. By
fixing a isomorphisms between the three copies of $\Delta_2$, we get a
projection from $(\Delta_2)^3$ onto $3\Delta_2$.
In the case of ideals 1, 2, 3, 4, and~5 the corresponding subcomplexes
project to tilings of $3\Delta_2$. In these cases, there is
a monomial ideal in the orbit which is in the toric Hilbert scheme of
 $\Delta_2\times \Delta_2$, and the corresponding triangulation of 
 $\Delta_2\times \Delta_2$
is related to the tiling of $3\Delta_2$ by the Cayley trick. 
The
tilings in~\cite[Figure~5]{Santos} correspond to the monomial ideals
1, 3, 4, 5, and~2 in this order. Each triangulation is regular, 
and the ideals are smoothable, even in
the toric Hilbert scheme. Here, the toric Hilbert is the
subscheme of $H_{3,3}$ obtained by fixing the 
Hilbert function of $I_2(X)$ with respect to the finer grading
given by both row degrees and column degrees.
For details, references and further information see
\cite[\S 2]{HS} and \cite[Theorem~2]{Santos}.

\begin{figure}
\begin{centering}
\includegraphics{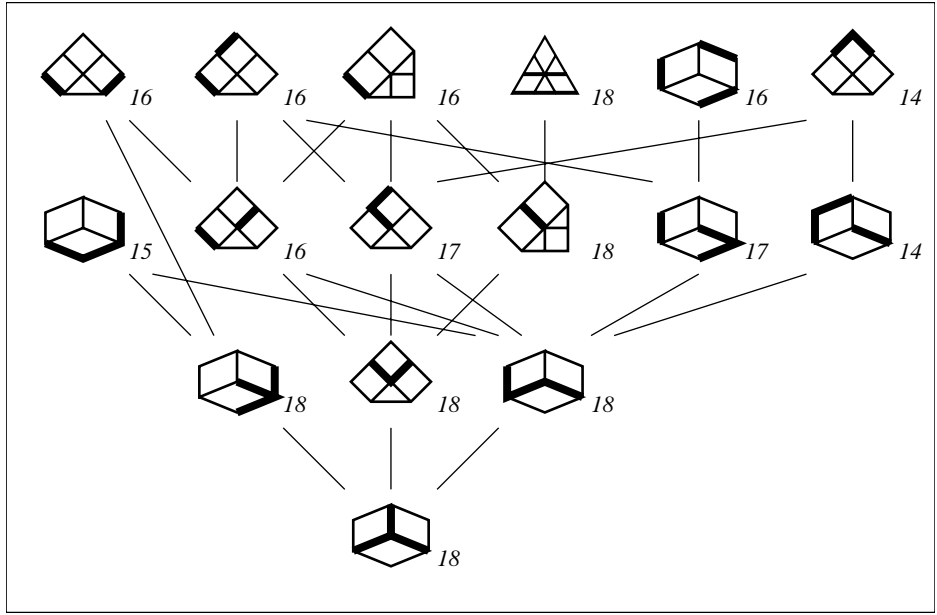}
\par\end{centering}
\caption{Partial ordering of the monomial ideals on $H_{3,3}$}
\label{fig:monomial-poset}
\end{figure}

\section{Deligne schemes and their special fibers}

The original motivation which started this project was a discussion
with Annette Werner about tropical convexity, and its connection to affine buildings
and moduli of hyperplane arrangements as developed by
Keel and Tevelev \cite{KT}. Our aim was to understand
the Deligne schemes of \cite[\S 1]{KT} and their special fibers in 
the concrete language of combinatorial commutative algebra.
We found that the multigraded Hilbert scheme $H_{d,n}$
offers a suitable framework for studying Deligne schemes
and their arithmetic. In this section we briefly discuss the set-up
and the connection to the combinatorial results in \cite{BY, JSY}.
We plan to pursue this further in a joint project with Annette Werner.

Let $K$ be an algebraically closed field with a non-trivial non-archimedean
absolute value, let $k$ be the residue field of $K$,
and $R$ the valuation ring of $K$.
In computational studies (such as~\cite{JSY}) we usually relax
the requirement that $K$ be algebraically closed,
and we work with the Gr\"obner-friendly
scenario $\, K = \Q(z)$, $R = \Q[z]$ and $k= \Q$.
Let $\mathcal{B}$ denote the Bruhat-Tits building associated
with the group ${\rm PGL}(d)$ over $K$ as defined in \cite{JSY, KT}.
The building~$\mathcal{B}$ is an infinite simplicial complex of dimension
$d-1$ whose vertices are the equivalence classes of
$R$-submodules of $K^d$ having maximal rank $d$.

Let $Y = \{Y_1,\ldots,Y_n\}$ be a finite set of vertices of the 
affine building~$\mathcal{B}$. Following \cite[Definition 1.8]{KT},
we let $\mathbb{S}_Y$ denote the corresponding join of projective spaces
over $R$, and we write $S_Y$ for its special fiber over $k$.
In the special case when $Y$ is a convex subset of~$\mathcal{B}$,
a classical result of Mustafin~\cite{Mus} states that $\mathbb{S}_Y$
is semi-stable over $R$, which implies that
$S_Y$ has smooth irreducible components with normal crossings.
In this section we allow $Y$ to be any finite set of vertices -- not
necessarily convex -- of the building~$\mathcal{B}$.
Following \cite[1.10]{KT}, we shall call $\mathbb{S}_Y$ the
{\em Deligne scheme} of the subset~$Y \subset \mathcal{B}$.

We now describe the Deligne scheme $\,\mathbb{S}_Y$ 
and its special fiber $S_Y$ in concrete terms.
The configuration $Y$ is represented by $(Y_1, Y_2, \ldots, Y_n)$, an 
$n$-tuple  of invertible $d \times d$-matrices 
with entries in the field $K$. This data is the input
for the algorithm of \cite{JSY} which computes the
convex hull of $Y$ in $ \mathcal{B}$.
Let $I_2(X)$ be the ideal of $2 {\times} 2$-minors
of a $d {\times} n$-matrix of unknowns.
We consider the transformed ideal $Y  \circ I_2(X)$ in $K[X]$,
and we intersect it with $R[X]$:
\begin{equation}
\label{GetSpecFib}
 \mathbb{I}_Y \,\,\, = \,\,\, (Y \circ I_2(X)) \,\cap \, R[X] . 
 \end{equation}
We call $\mathbb{I}_Y \subset R[X]$ the {\em Deligne ideal} of the 
point configuration $\,Y \subset \mathcal{B}$.
The image of $\mathbb{I}_Y$ under the specialization
$R \rightarrow k$ is denoted $I_Y \subset k[X]$. We
call $I_Y$ the {\em special fiber ideal} of $Y$.
This nomenclature is justified as follows.

\begin{remark}
\label{rem:deligne}
The Deligne scheme $\,\mathbb{S}_Y$  coincides with
the subscheme of
$\,(\PP^{d-1}_R)^n \,$ defined by the Deligne ideal 
$\mathbb{I}_Y$, and its special fiber
$S_Y$ coincides with the subscheme of $(\PP^{d-1}_k)^n$
defined by the special fiber ideal $I_Y$.
\end{remark}

Remark \ref{rem:deligne}  implies that the Deligne scheme $\mathbb{S}_Y$ is
a point in the multigraded Hilbert scheme $H_{d,n}(R)$ over the valuation ring $R$,
and its special fiber $S_Y$ is a point in $H_{d,n}(k)$ over its residue field $k$.
Thus our study in Sections~2 to~5 is relevant for Deligne schemes.
In particular, Theorem \ref{thm:radical} implies:

\begin{cor}
The special fiber $\,S_Y$ of the Deligne scheme $\,\mathbb{S}_Y$ is reduced.
\end{cor}
 
The formula (\ref{GetSpecFib}) translates into the following Gr\"obner-based
algorithm for computing Deligne schemes and their special fibers
when $K = \Q(z), R = \Q[z], k = \Q$.
The input data is an $n$-tuple  $Y(z)$ of
invertible $d \times d$-matrices whose entries are
rational functions in one variable $z$. 
The Deligne ideal $\mathbb{I}_{Y(z)}$ is an ideal
in the polynomial ring $\Q[z,X]$. It is computed as follows. 
We replace the $j$-th column of the matrix $X = (x_{ij})$
by its left multiplication with the $j$-th input matrix $Y_j(z)$.
We then form the $2 \times 2$-minors of the resulting matrix
and multiply each generator by a power of $z$ to 
get an ideal~$L$ in $\Q[z,X]$. The Deligne ideal is then 
obtained by saturation as follows:
\begin{equation}
\label{sat1}
 \mathbb{I}_{Y(z)} \,\, = \,\, \bigl( L : \langle z \rangle^\infty \bigr). 
\end{equation}
The special fiber ideal is obtained by setting $z$ to zero in the Deligne ideal:
\begin{equation}
\label{sat2}
 I_{Y(z)} \,\, = \,\, \mathbb{I}_{Y(z)}|_{z=0}. 
 \end{equation}
 
 This algorithm generalizes the construction
of Block and Yu \cite{BY} which concerns the special 
case when the configuration $Y(z)$ lies in one apartment 
of the Bruhat-Tits building~$\mathcal{B}$.
Algebraically, this means that the $Y_j(z)$ are diagonal matrices,
and, geometrically, the apartment of $\mathcal{B}$ is identified
with the standard triangulation of tropical projective space $\TP^{d-1}$.  
If the diagonal entries of the $Y_j(z)$ are monomials with 
sufficiently generic exponents, then 
(\ref{sat1})-(\ref{sat2}) simply amounts  to a
Gr\"obner basis computation for the ideal~$I_2(X)$.
Block and Yu \cite{BY} consider the Alexander dual 
of the initial monomial ideal of $I_2(X)$, and they showed
that minimal free resolution of that Alexander dual is
cellular. It represents the convex hull of $Y(z)$ in 
$\TP^{d-1}$, and hence in $ \mathcal{B}$.
We conjecture that this method can be adapted to
compute the convex hull of $Y(z)$ even if
the matrices $Y_j(z)$ do not lie in a common apartment of $\mathcal{B}$.

We note that the algorithm (\ref{sat1})-(\ref{sat2}) is not very practical for computing
the special fiber ideal $I_{Y(z)}$. We found that  running the saturation step
(\ref{sat1}) in naive manner in {\tt Macaulay 2}  is too slow for interesting values of $d$ and $n$.

Instead, we propose the following approach. If the 
$Y(z)$ are ``sufficiently generic,'' 
then we replace each $Y_j(z)$ by a ``nearby''
matrix of the special form
\begin{equation}
\label{approx1}
\,\,Y_j(z) \quad \approx \quad  {\rm diag}(z^{w_{1j}}, \ldots, z^{w_{dj}}) \cdot A_i\, . 
\end{equation}
Here the weights $w_{ij}$ are generic rational numbers that
specify a term order~$>$ on the polynomial ring $\Q[X]$, and
$A = (A_1,\ldots, A_n)$ is a tuple of invertible matrices over $\Q$.
The precise meaning of the approximation (\ref{approx1}) is that
\begin{equation}
\label{approx2}
 I_{Y(z)}\,\, = \,\,  {\rm in}_>(A \circ I_2(X)). 
 \end{equation}
 When this holds 
   the  special fiber $S_{Y(z)}$ of the Deligne scheme 
 $\mathbb{S}_{Y(z)}$ is given by
a squarefree monomial ideal $I_{Y(z)}$.
The point is that the right hand side of
 (\ref{approx2}) can be computed much faster in practice than evaluating
 (\ref{sat1}). It amounts to  computing a Gr\"obner basis
of the transformed ideal $A \circ I_2(X)$ in the polynomial ring $\Q[X]$.
When the $A_i$ are  diagonal matrices over $\Q$ this
is precisely the algorithm of Block and Yu \cite{BY} for convex hulls in $\TP^{d-1}$.

\bigskip

\noindent {\bf Acknowledgements.}
Both authors were supported by the U.S.~National 
Science Foundation (DMS-0354321, DMS-0456960 and DMS-0757236).
We thank Michel Brion, Aldo Conca and 
Annette Werner for helpful comments.

\medskip

\noindent {\bf Authors' addresses:}

\smallskip

\noindent Dustin Cartwright and Bernd Sturmfels, Dept.~of Mathematics,
University of California, Berkeley, CA 94720, USA,
\url{{dustin,bernd}@math.berkeley.edu}


\begin{thebibliography}{99}

\bibitem{AS} K.~Altmann and B.~Sturmfels:
The graph of monomial ideals, {\em Journal of Pure and Applied Algebra}
{\bf 201} (2005) 250--263.

\bibitem{BY}
F.~Block and J.~Yu:
Tropical convexity via cellular resolutions,
{\em Journal of Algebraic Combinatorics} {\bf 24} (2006) 103--114.

\bibitem{Brion}
M.~Brion: Group completions via Hilbert schemes,
{\em Journal of Algebraic Geometry} {\bf 12} (2003) 605--626. 

\bibitem{BrionMult}
M.~Brion: Multiplicity-free subvarieties, {\em Contemporary Math.} {\bf 331}
(2003) 13--23.

\bibitem{Conca}
A.~Conca: Linear spaces, transversal polymatroids and ASL domains,
{\em Journal of Algebraic Combinatorics}
{\bf 25} (2007) 25--41.

\bibitem{Eis}
D.~Eisenbud: {\em Commutative Algebra With a View Toward Algebraic Geometry},
Graduate Texts in Mathematics {\bf 150}, Springer, New York, 1995.

\bibitem{GHZ} V.~Guillemin, T.~Holm and C.~Zara:
A GKM description of the equivariant cohomology ring of a homogeneous space,
{\em Journal of Algebraic Combinatorics} {\bf 23} (2006) 21--41.

\bibitem{HS}
M.~Haiman and B.~Sturmfels: Multigraded Hilbert schemes,
{\em Journal of Algebraic Geometry} {\bf 13} (2004) 725--769.
 
\bibitem{JSY}
M.~Joswig, B.~Sturmfels and J.~Yu:
Affine buildings and tropical convexity, 
{\em Albanian Journal of Mathematics} {\bf 1} (2007) 187--211. 

\bibitem{GM}
A.~Gibney and D.~Maclagan:
Equations for Chow and Hilbert quotients, {\tt arXiv:0707.1801}.

\bibitem{KT}
S.~Keel and J.~Tevelev: Geometry of Chow quotients of Grassmannians, 
{\em Duke Mathematical Journal} {\bf 134} (2006) 259--311.

\bibitem{Laf} L.~Lafforgue: Pavages des simplexes,
sch\'emas de graphes recoll\'es et compactification
des ${\rm PGL}_r^{n+1}/{\rm PGL}_r$,
{\em Invent.~math.} {\bf 136} (1999) 233--271;
Erratum in
{\em Invent.~math.} {\bf 145} (2001) 619--610.

\bibitem{MS}
E.~Miller and B.~Sturmfels:
{\em Combinatorial Commutative Algebra},
Graduate Texts in Mathematics {\bf 227}, Springer, New York, 2004.

\bibitem{Mus} G.A.~Mustafin:
Non-archimedean uniformization, {\em Math.~USSR Sbornik} {\bf 34} (1978) 187--214.

\bibitem{Santos}
F.~Santos: The Cayley trick and triangulations of products of simplices, In
 {\em Integer Points in Polyhedra - Geometry, Number Theory, Algebra, Optimization}, 
(eds.~A.~Barvinok, M.~Beck, C.~Haase, B.~Reznick, V.~Welker),
 Contemporary Mathematics
 {\bf 374}, AMS, Providence, 2005.

 \bibitem{Stanley}
R.~Stanley: {\em Combinatorics and Commutative Algebra},
Second Edition,
 Birkh\"auser, Boston, 2004.

\bibitem{Th}
M.~Thaddeus:  Complete collineations revisited,
{\em Mathematische Annalen} {\bf 315} (1999) 469--495.
\end{thebibliography}
\end{document}